\documentclass{amsart}
\usepackage[margin=1.4in]{geometry}
\usepackage{amsmath,amsthm,amssymb,comment,enumerate}
\usepackage{graphicx}
\usepackage{color}

\theoremstyle{plain}
\newtheorem{thm}{Theorem}[section]
\newtheorem{cor}[thm]{Corollary}
\newtheorem{lem}[thm]{Lemma}
\newtheorem{prop}[thm]{Proposition}

\newtheorem{thmx}{Theorem}

\theoremstyle{definition}

\theoremstyle{remark}
\newtheorem{rem}[thm]{Remark}

\numberwithin{equation}{section}

\begin{document}
\title{Semiclassical bounds for spectra of biharmonic operators}%
\author{Davide Buoso}%
\author{Luigi Provenzano}%
\author{Joachim Stubbe}%

\address{Davide Buoso, Dipartimento di Scienze e Innovazione Tecnologica (DiSIT), Università degli Studi del Piemonte Orientale ``A. Avogadro'', Viale Teresa Michel 11, 15121 Alessandria (ITALY). E-mail: {\tt davide.buoso@uniupo.it}}%
\address{Luigi Provenzano, Dipartimento di Scienze di Base e Applicate per l'Ingegneria, Sapienza Universit\`a di Roma, Via Antonio Scarpa 16, 00161 Roma, Italy. E-mail: {\tt luigi.provenzano@uniroma1.it} }
\address{Joachim Stubbe, EPFL, SB MATH SCI-SB-JS, Station 8, CH-1015 Lausanne, Switzerland. E-mail: {\tt joachim.stubbe@epfl.ch}}%


\begin{abstract}
We provide complementary semiclassical bounds for the Riesz means $R_1(z)$ of the eigenvalues of various biharmonic operators, with a second term in the expected power of $z$. The method we discuss makes use of the averaged variational principle (AVP), and yields two-sided bounds for individual eigenvalues, which are semiclassically sharp. The AVP also yields comparisons with Riesz means of different operators, in particular Laplacians.


\end{abstract}
\maketitle
\setcounter{page}{1}
\tableofcontents


\section{Introduction} 

Let $\Omega\subset\mathbb{R}^d$ be a bounded domain with boundary $\partial\Omega$.
We consider the eigenvalue problem for the biharmonic operator with various boundary conditions:
\begin{equation}\label{Biharmonic-ev-problem}
\left\{\begin{array}{ll}
\Delta^2 u= \omega u, & \text{on }\Omega,\\
A_1(u)=A_2(u)= 0, & \text{ on }\partial\Omega.
\end{array}\right.
\end{equation}
The biharmonic operator $\Delta^2=\Delta\Delta$ is the first iteration of the Laplace operator $-\Delta$, and $A_1(u),A_2(u)$ represent two linear operators which we shall specify for each problem. These operators are generated from self-adjoint representations of various quadratic forms defined on a suitable dense closed subspace of the Sobolev space $H^2(\Omega)$, see Section 2.

The interest of studying problem \eqref{Biharmonic-ev-problem} is motivated by several applications as the modelling of vibrations of a thin elastic plate subject to different constraints or the static loading of a slender beam, and models for suspension bridges. We refer the reader to \cite{ciardest, ciarlet, gander, gazponti, landau, love, sweers2009} for more details on the applications related with problem \eqref{Biharmonic-ev-problem}.

We always suppose that the spectrum of \eqref{Biharmonic-ev-problem} consists of an ordered
sequence of eigenvalues $\omega_j$ tending to infinity,
\begin{equation*}
  0\leq\omega_1\leq\omega_2\leq\omega_3\leq \cdots
\end{equation*}
This assumption holds, for example, when $\Omega$ has finite Lebesgue measure and the boundary conditions in \eqref{Biharmonic-ev-problem} are given by the so-called Dirichlet boundary conditions
$$
A_1(u)=u,\ \ A_2(u)=|\nabla u|,
$$
(where $\nabla$ denotes the grandient operator) emerging from the study of the oscillations of a clamped plate. For other boundary conditions and precise definitions we refer to Section 2.

An important issue in the spectral theory of partial differential operators is the asymptotic expansion of the eigenvalues $\omega_j$ as $j\to\infty$ and eigenvalue bounds in terms of the asymptotic expansion, called semiclassical estimates, which is the main subject of the present paper for the eigenvalue problem \eqref{Biharmonic-ev-problem}.
To this end, it is convenient to consider the counting function
\begin{equation*}
  N(z)={\rm Card}\{\omega_j: \omega_j<z,\ \omega_j {\rm \ is\ an\ eigenvalue}\},
\end{equation*}
and, in a tradition due to Berezin \cite{Ber}, Riesz means,
\begin{equation*}
R_\sigma(z) =\sum_{j}(z-\omega_{j})_{+}^\sigma,
\end{equation*}
with $\sigma>0$ (here $x_{+}$ denotes the positive part of $x$).
$N(z)$ can be interpreted as the limit of $R_\sigma(z)$ when $\sigma \to 0$. The Riesz means $R_\sigma(z)$ are related to $N(z)$ via the integral transform
\begin{equation}
\label{inttransf}
R_\sigma(z)=\sigma \int_0^\infty(z-t)_+^{\sigma-1}N(t)dt,
\end{equation}
and in particular the behavior of $\omega_j$ as $j\to\infty$ is given by the asymptotic expansion of the counting function $N(z)$ as $z\to\infty$.
There is a large literature dealing with the asymptotic expansion of the counting function or other spectral quantities, we refer to the books by Ivrii \cite{monsterbook} and Safarov and Vassiliev \cite{safvas}, that present the state of the art as well as the key references.

The leading term in the asymptotic expansion is known as the Weyl limit, going back to the fundamental work of H.\ Weyl \cite{weyl} on the asymptotic behavior of Dirichlet Laplacian eigenvalues
$$
\left\{\begin{array}{ll}
-\Delta u= \omega u, & \text{on }\Omega,\\
u=0, & \text{ on }\partial\Omega.
\end{array}\right.
$$
It is now known that the Weyl limit depends on the principal symbol of the partial differential operator which is connected to the Fourier transform and equals $|p|^2$ for the Laplace operator, and $|p|^4$ for biharmonic operator.

We may summarize the Weyl law for an operator with principal symbol $|p|^{2m}$ as
\begin{equation}
\label{weyllaw}
\lim_{z\to\infty}\frac{N(z)}{z^{\frac d {2m}}}=(2\pi)^{-d}\int_\Omega\int_{\mathbb R^d}(1-|p|^{2m})_+^0dpdx=(2\pi)^{-d}B_d|\Omega|,
\end{equation}
where the right hand side corresponds to the normalized phase space volume of the operator. Here $m=1,2$, but \eqref{weyllaw} remains true for higher iterations of the Laplacian on a bounded domain $\Omega$ under suitable boundary conditions (see e.g., \cite{safvas}). Here $B_d=\frac{\pi^{d/2}}{\Gamma(1+d/2)}$ is the volume of the $d$-dimensional unit ball. The equivalent statement for the eigenvalues $\omega_j$ is
\begin{equation}
\label{weyllaweig}
\lim_{j\to\infty}\frac{\omega_j}{j^{\frac{2m}{d}}}=C_d^m|\Omega|^{-\frac{2m}{d}},
\end{equation}
where $C_d=(2\pi)^2B_d^{-\frac 2 d}$ is the so-called classical constant.

The Weyl law \eqref{weyllaw} or \eqref{weyllaweig} is of striking simplicity. The limit depends only on the volume of the domain and a universal dimensional constant and is independent of the boundary conditions. In particular, we infer from the Weyl law that at least asymptotically the eigenvalues of the biharmonic problem \eqref{Biharmonic-ev-problem} equal the squares of Laplacian eigenvalues.

One may then ask whether the counting function is bounded by its Weyl law, that is whether it is possible to establish sharp semiclassical bounds of the type
$$
N(z)z^{-\frac d {2m}}\le (2\pi)^{-d}B_d|\Omega|,\quad{\rm or}\quad N(z)z^{-\frac d {2m}}\ge (2\pi)^{-d}B_d|\Omega|,
$$
for all $z\ge 0$. Even in the simpler case of Laplacian eigenvalues ($m=1$) this is, apart from special domains, still an open problem known as Polya's conjecture where it is conjectured that the first inequality should hold for Dirichlet boundary conditions.
However, for Riesz means $R_1(z)=\int_0^zN(t)dt$, sharp bounds have been obtained both for Laplace and biharmonic operators since there are plenty of variational techniques which can be applied, see e.g., Berezin \cite{Ber}, Li-Yau \cite{LiYau}, Kr\"oger, \cite{Kro} and Laptev \cite{Lap1997, Lap2012}.

Combining the Weyl law \eqref{weyllaw} for $N(z)$ and the integral relation \eqref{inttransf}, one obtains
\begin{equation*}
\lim_{z\to\infty}\frac{R_1(z)}{z^{1+\frac d {2m}}}=(2\pi)^{-d}\int_\Omega\int_{\mathbb R^d}(1-|p|^{2m})_+dpdx=\frac{2m}{2m+d}(2\pi)^{-d}B_d|\Omega|,
\end{equation*}
and the corresponding sharp semiclassical bounds are of the form
$$
R_1(z)\le \frac{2m}{2m+d}(2\pi)^{-d}B_d|\Omega|z^{1+\frac d {2m}},\quad{\rm or} \quad R_1(z)\ge \frac{2m}{2m+d}(2\pi)^{-d}B_d|\Omega|z^{1+\frac d {2m}}.
$$
When $m=1$ (Dirichlet Laplacian eigenvalues) the first inequality is the celebrated Berezin-Li-Yau bound and when $m=2$ it has been shown for biharmonic Dirichlet eigenvalues by Levine and Protter \cite{LePr}. 

The effect of boundary conditions on the spectrum is already seen in the second term of the asymptotic expansion, i.e., that following the Weyl law. As shown in \cite{monsterbook,safvas}, at least for smooth domains $\Omega\subset\mathbb R^d$ there is a two terms asymptotic expansion of the form
\begin{equation}
\label{twoterms}
N(z)=(2\pi)^{-d}B_d|\Omega|z^{\frac d {2m}}+a_{d,m}|\partial\Omega|z^{\frac{d-1}{2m}}+o\left(z^{\frac{d-1}{2m}}\right),
\end{equation}
where $a_{d,m}$ is a real constant depending on the dimension $d$, the order $m$ of the differential operator (where as before $m=1$ corresponds to Laplacian eigenvalues and $m=2$ to biharmonic eigenvalues), and on the boundary conditions. Applying these techniques we compute \eqref{twoterms} for various boundary conditions of the biharmonic eigenvalue problem \eqref{Biharmonic-ev-problem}.

We remark that all the classical strategies to get two-term expansions for eigenvalues of elliptic operator, as shown in \cite{monsterbook, safvas}, involve the extensive use of microlocal analysis, that requires a number of regularity conditions on the domain $\Omega$ that are not yet well understood in simple geometrical terms. However, recently Frank and Larson \cite{franklarson} (see also \cite{frankgei1, frankgei2}) have proved a two-term expansion for Riesz means of Laplacian eigenvalues without using microlocal analysis and with low regularity assumptions on $\Omega$.

The asymptotic expansion \eqref{twoterms} suggests to look for bounds of $N(z)$ (or $R_1(z)$) in terms of \eqref{twoterms}. 
In this work we show that in some circumstances the Levine-Protter bound can be reversed, so that there is a kind of ``Kr\"oger'' lower bound for Dirichlet Bilaplacian Riesz means (i.e., an upper bound for eigenvalue averages). Reversing the inequalities requires lower-order correction terms, which, as will be seen, include information about the boundary of the domain. As in \cite{HaSt16} an essential tool will be the averaged variational principle first introduced in \cite{HaSt14} (see also \cite{EHIS}), which gives an efficient derivation of Kr\"oger's inequality and has been used to derive various other upper bounds for averages of eigenvalues.

Moreover, since the techniques for the two term asymptotic expansion do not apply to the eigenvalue problem \eqref{Biharmonic-ev-problem} on an interval (that is, $d=1$) we study separately the one dimensional problems and exhibit a remarkable common similarity of the different spectra, see Section \ref{biharmonic1d}. In particular, two terms asymptotics for one dimensional problems are a consequence of asymptotically sharp upper and lower bounds of the corresponding Riesz means which we provide in Section \ref{biharmonic1d}

The paper is organized as follows. In Section 2 we introduce the biharmonic eigenvalue problems we study and we present the main results of the paper. In Section 3 we prove some inequalities between biharmonic eigenvalues and then compute the respective semiclassical asymptotic expansions. Section 4 is dedicated to the semiclassical estimates for Dirichlet Bilaplacian eigenvalues, while Navier and Kuttler-Sigillito eigenvalues are treated in Section 5. Section 6 contains a few remarks on Neumann Bilaplacian eigenvalues. Finally, Section 7 is devoted to the one dimensional biharmonic eigenvalue problems.


\section{Biharmonic eigenvalue problems and main results}\label{sec:2}

In this section we introduce the eigenvalue problems of the form \eqref{Biharmonic-ev-problem} that we will study in the sequel and present the main results of the paper. Unless differently specified, we assume $\Omega\subset\mathbb R^d$ to be a bounded domain with Lipschitz boundary $\partial\Omega$.

In the following we will denote by $1_A$ the characteristic function of $A\subseteq\mathbb R^d$. For a function $f\in L^1(\mathbb R^d)$ we will denote by $\hat f(\xi)$ its Fourier transform  defined by $\hat f(\xi)=(2\pi)^{-d/2}\int_{\mathbb R^d}f(x)e^{i\xi\cdot x}dx$, and with abuse of notation, for a function $f\in H^2_0(\Omega)$ we will still denote by $\hat f(\xi)$ the Fourier transform of its extension by zero to $\mathbb R^d$. We will also denote by $B(x,R)$ the $d$-dimensional ball (in $\mathbb R^d$) of radius $R$ centered at the point $x$.

The Dirichlet Laplacian eigenvalue problem is
$$
\left\{\begin{array}{ll}
-\Delta u_j=\lambda_j u_j, & \text{in\ }\Omega,\\
u_j=0, & \text{on\ }\partial\Omega,
\end{array}\right.
$$
and the eigenvalues are variationally characterized by
$$
\lambda_j=\min_{\substack{V\subset H^1_0(\Omega) \\ {\rm dim\ }V=j}}\max_{u\in V\setminus \{0\}}\frac{\int_\Omega |\nabla u|^2}{\int_\Omega u^2}.
$$
The Neumann Laplacian eigenvalue problem is
$$
\left\{\begin{array}{ll}
-\Delta v_j=\mu_j v_j, & \text{in\ }\Omega,\\
\frac{\partial v_j}{\partial\nu}=0, & \text{on\ }\partial\Omega,
\end{array}\right.
$$
and the eigenvalues are variationally characterized by
$$
\mu_j=\min_{\substack{V\subset H^1(\Omega) \\ {\rm dim\ }V=j}}\max_{u\in V\setminus \{0\}}\frac{\int_\Omega |\nabla u|^2}{\int_\Omega u^2}.
$$

The biharmonic eigenvalue equation we will consider is
\begin{equation}
\label{equazione}
\Delta^2u=\omega u,\ \ \text{in\ }\Omega,
\end{equation}
complemented with four different sets of boundary conditions:
\begin{itemize}
\item Dirichlet boundary conditions:
\begin{equation}
\label{DBC}
u=\frac{\partial u}{\partial \nu}=0;
\end{equation}
\item Navier boundary conditions:
\begin{equation}
\label{IBC}
u=(1- a )\frac{\partial^2 u}{\partial \nu^2}+ a \Delta u=0;
\end{equation}
\item Kuttler-Sigillito boundary conditions:
\begin{equation}
\label{KBC}
\frac{\partial u}{\partial \nu}=\frac{\partial\Delta u}{\partial \nu}+(1- a ){\rm div}_{\partial\Omega}\left(\frac{\partial\ }{\partial \nu}\nabla_{\partial\Omega}u\right)=0;
\end{equation}
\item Neumann boundary conditions:
\begin{equation}
\label{NBC}
(1- a )\frac{\partial^2 u}{\partial \nu^2}+ a \Delta u=\frac{\partial\Delta u}{\partial \nu}+(1- a ){\rm div}_{\partial\Omega}\left(\frac{\partial\ }{\partial \nu}\nabla_{\partial\Omega}u\right)=0.
\end{equation}
\end{itemize}

Here $\nu$ is the outer unit normal vector defined on $\partial \Omega$, ${\rm div}_{\partial\Omega}$ and $\nabla_{\partial\Omega}$ are the tangential divergence and the tangential gradient on $\partial \Omega$, respectively, and $ a $ is the Poisson ratio, $ a \in(-(d-1)^{-1},1)$. Note that the quadratic form associated with all these problems is
\begin{equation}
\label{qf}
Q(u,v)=\int_\Omega(1- a )D^2u:D^2v+ a \Delta u\Delta v,
\end{equation}
but this form is set in $H^2_0(\Omega)$ for Dirichlet boundary conditions, in $H^2(\Omega)\cap H^1_0(\Omega)$ for Navier boundary conditions, in $H^2_\nu(\Omega)=\{u\in H^2(\Omega): \frac{\partial u}{\partial\nu}=0 {\rm \ on\  }\partial\Omega\}$ for Kuttler-Sigillito boundary conditions, and in $H^2(\Omega)$ for Neumann boundary conditions. In particular, the Dirichlet problem does not see the Poisson ratio, as
$$
\int_\Omega D^2u:D^2v=\int_\Omega\Delta u\Delta v
$$
for any $u,v,\in H^2_0(\Omega)$. Here and in the sequel, the Frobenius product is defined as
$$
D^2u:D^2v=\sum_{\alpha,\beta=1}^d\frac{\partial^2u}{\partial x_\alpha\partial x_\beta}\frac{\partial^2v}{\partial x_\alpha\partial x_\beta}.
$$
Furthermore, we will denote by $U_j,\Lambda_j$ the eigenfunctions and the eigenvalues of the Dirichlet problem \eqref{equazione}, \eqref{DBC}. Similarly, we will use $\tilde U_j,\tilde \Lambda_j(a)$ for the Navier problem \eqref{equazione}, \eqref{IBC}, $\tilde V_j,\tilde M_j(a)$ for the Kuttler-Sigillito problem \eqref{equazione}, \eqref{KBC}, and $V_j,M_j(a)$ for the Neumann problem \eqref{equazione}, \eqref{NBC}. We will not write explicitly the dependence on the Poisson ratio $a$ for eigenfunctions, but we will for eigenvalues (with the exception of Dirichlet eigenvalues that do not depend on $a$). When we consider these problems in general without specifying the boundary conditions, we will use instead $u,\omega$ as a generic eigenfunction with its associated eigenvalue.

Note that the eigenvalues can be characterized via the minimax formulation as
$$
\Lambda_j=\min_{\substack{V\subset H^2_0(\Omega) \\ {\rm dim\ }V=j}}\max_{u\in V\setminus \{0\}}\frac{\int_\Omega (\Delta u)^2}{\int_\Omega u^2},
$$

$$
\tilde \Lambda_j(a)=\min_{\substack{V\subset H^2(\Omega)\cap H^1_0(\Omega) \\ {\rm dim\ }V=j}}\max_{u\in V\setminus \{0\}}\frac{\int_\Omega(1- a )|D^2u|^2+ a (\Delta u)^2,
}{\int_\Omega u^2},
$$

$$
\tilde M_j(a)=\min_{\substack{V\subset H^2_\nu(\Omega) \\ {\rm dim\ }V=j}}\max_{u\in V\setminus \{0\}}\frac{\int_\Omega(1- a )|D^2u|^2+ a (\Delta u)^2,
}{\int_\Omega u^2},
$$

 and
$$
M_j(a)=\min_{\substack{V\subset H^2(\Omega) \\ {\rm dim\ }V=j}}\max_{u\in V\setminus \{0\}}\frac{\int_\Omega(1- a )|D^2u|^2+ a (\Delta u)^2,
}{\int_\Omega u^2}.
$$

It is worth observing that, when $a=1$, the Navier problem \eqref{equazione}, \eqref{IBC} becomes
\begin{equation}
\label{pure_navier}
\left\{\begin{array}{ll}
\Delta^2 \tilde U=\tilde\Lambda(1) \tilde U , & \text{in\ }\Omega,\\
\tilde U=\Delta \tilde U=0, & \text{on\ }\partial\Omega.
\end{array}\right.
\end{equation}
If the domain $\Omega$ is only Lipschitz, in principle the quadratic form \eqref{qf} is not coerive in $H^2(\Omega)\cap H^1_0(\Omega)$ and the spectrum of problem \eqref{pure_navier} may be not variationally characterizable. However, the form is coercive as soon as $\Omega$ also satisfies the so-called uniform outer ball condition (see \cite{adolfsson}; see also \cite[Section 2.7]{ggs}). In particular, in this case the domain of the Dirichlet Laplacian is precisely $H^2(\Omega)\cap H^1_0(\Omega)$ and the following identification becomes immediate
\begin{equation*}
\tilde U_j =u_j,\ \ \tilde\Lambda_j(1)=\lambda_j^2,
\end{equation*} 
for all $j\in\mathbb N$. Note that, in the literature, problem \eqref{pure_navier} is known to be the classical Navier problem, whereas problem \eqref{equazione}, \eqref{IBC} is a more recent generalization (see also \cite{buoso16,ggs} and the references therein for a discussion on the physical meaning of the problem).  Analogously, when $a=1$, the Kuttler-Sigillito problem \eqref{equazione}, \eqref{KBC} becomes
\begin{equation}
\label{pure_ks}
\left\{\begin{array}{ll}
\Delta^2 \tilde V=\tilde M(1) \tilde V , & \text{in\ }\Omega,\\
\frac{\partial \tilde V}{\partial\nu}=\frac{\partial \Delta\tilde V}{\partial\nu}=0, & \text{on\ }\partial\Omega.
\end{array}\right.
\end{equation}
Again, if $\Omega$ is only Lipschitz, the problem \eqref{pure_ks} may not be variationally characterizable, but for $\Omega$ smooth enough we recover that the domain of the Neumann Laplacian is precisely $H^2_\nu(\Omega)$ and then
\begin{equation*}
\tilde V_j =v_j,\ \ \tilde M_j(1)=\mu_j^2,
\end{equation*} 
for all $j\in\mathbb N$. We point out that problem \eqref{equazione}, \eqref{KBC} has not been widely studied yet, although it is appearing as an important problem in a number of different situations (see e.g., \cite{buosokennedy, lamproz1, lamproz2}). We remark though that it was first stated as a Steklov-type problem by Kuttler and Sigillito in \cite{kutsig}.

On the other hand, the Neumann problem \eqref{equazione}, \eqref{NBC} with $a=1$ becomes instead
\begin{equation*}
\left\{\begin{array}{ll}
\Delta^2 V=M(1) V , & \text{in\ }\Omega,\\
\Delta V=\frac{\partial\Delta V}{\partial\nu}=0, & \text{on\ }\partial\Omega,
\end{array}\right.
\end{equation*}
so that the boundary conditions do not satisfy the complementing conditions (see e.g., \cite{ggs}), and in particular it has a kernel consisting of the harmonic functions in $H^2(\Omega)$, which is infinite dimensional when $d\ge2$. It was shown in \cite{provneu} that the remaining part of the spectrum consists of the eigenvalues of the biharmonic Dirichlet problem \eqref{equazione}, \eqref{DBC}.

The main results of the present paper are inequalities related to the eigenvalues of problem \eqref{Biharmonic-ev-problem}, both for the eigenvalues and for Riesz means $R_1(z)$. To this end, we first provide inequalities between the eigenvalues of the different problems (see Theorem \ref{3.1}).

\begin{thmx} \label{thma}
The following inequalities hold.
\begin{itemize}
\item For any $j\in\mathbb N$, and for any $ a \in(-(d-1)^{-1},1)$,
\begin{equation*}
M_j\leq\tilde M_j, \tilde\Lambda_j\leq\Lambda_j.
\end{equation*}
\item For any $j\in\mathbb N$,
\begin{equation*}
\lambda_j^2\le\Lambda_j.
\end{equation*}
\item If in addition $\Omega$ is convex, then for any $j\in\mathbb N$, and for any $ a \in(-(d-1)^{-1},1)$,
\begin{equation*}
M_j(a)\le\mu_j^2.
\end{equation*}
\end{itemize}
\end{thmx}

In order to understand when our bounds are sharp with respect to the semiclassical asymptotic expansion, we first compute it for all the boundary conditions (see Theorem \ref{asymptotal}).
We remark that, while our assumptions are enough to ensure the validity of the first term in expansions \eqref{semiclassicalcounting} and \eqref{semiclassicalriesz} (see e.g., \cite{lapidus} and the references therein), the derivation of the second term requires additional regularity on the domain $\Omega$. In particular, the domain has to be at least piecewise $C^\infty$ and the so-called {\it nonperiodicity} and {\it nonblocking} conditions have to be satisfied. We refer to \cite[Chapter 1]{safvas} for the description of all the necessary smoothness conditions, in particular to \cite[Definition 1.3.7]{safvas} for the definition of nonperiodicity condition and to \cite[Definition 1.3.22]{safvas} for the definition of nonblocking condition.


\begin{thmx}\label{intro_exp}
\label{thmb}
Let $\Omega$ be piecewise $C^{\infty}$ satisfying the nonperiodicity and nonblocking conditions, and let $d\ge 2$. We have
\begin{equation}
\label{semiclassicalcounting}
N(z)=(2\pi)^{-d}B_d|\Omega|z^{\frac d 4}+c_1 z^{\frac{d-1}4}+o(z^{\frac{d-1}4}),
\end{equation}
where the geometrical constant $c_1$ involving the measure of the boundary is given by \eqref{c1dir}--\eqref{c1neu}. In particular,
\begin{equation}
\label{semiclassicalriesz}
R_1(z)=\frac{4}{d+4}(2\pi)^{-d}B_d|\Omega|z^{\frac{d+4}{4}}+\frac{4 c_1}{d+3}z^{\frac{d+3}{4}}+o(z^{\frac{d+3}{4}}).
\end{equation}
\end{thmx}

Our third main result concerns lower bounds for Riesz means of Dirichlet eigenvalues (see Theorem \ref{dirichlet_bilaplacian_thm_general}).

\begin{thmx}\label{dirichlet_bilaplacian_thm_general0}
Let $\Omega$ be a domain in $\mathbb R^d$ of finite measure. Then for any $\phi\in H^2_0(\Omega)\cap L^{\infty}(\Omega)$ and $z>0$ the following inequality holds
\begin{multline*}
\sum_j\left(z-\Lambda_j\right)_+\|\phi U_j\|_2^2\geq\frac{4}{d+4}(2\pi)^{-d} B_d\|\phi\|_2^2\left(z-\frac{\|\Delta\phi\|_2^2}{\|\phi\|_2^2}\right)^{\frac{d}{4}+1}_+\\
-2(2\pi)^{-d} B_d\|\nabla\phi\|_2^2\left(z-\frac{\|\Delta\phi\|_2^2}{\|\phi\|_2^2}\right)^{\frac{d}{4}+\frac{1}{2}}_+.
\end{multline*}
Moreover, for all positive integers $k$
\begin{equation*}
 \frac{1}{k}\sum_{j=1}^k\Lambda_j\leq\frac{d}{d+4}C_d^2\left(\frac{k}{|\Omega|}\right)^{\frac{4}{d}}\rho(\phi)^{-\frac{4}{d}}
+2\frac{\|\nabla\phi\|_2^2}{\|\phi\|_2^2}C_d\left(\frac{k}{|\Omega|}\right)^{\frac{2}{d}}\rho(\phi)^{-\frac{2}{d}}
+\frac{\|\Delta\phi\|_2^2}{\|\phi\|_2^2},
\end{equation*}
for $\rho(\phi)<1$, where $\rho(\phi)$ is defined in \eqref{rhophi}.
\end{thmx}

Theorem \ref{dirichlet_bilaplacian_thm_general0} implies two-terms, asymptotically sharp lower bounds for Riesz means compatible with the two-term asymptotics given by Theorem \ref{intro_exp} (see Theorem \ref{main_upper_dirichlet}). Analogous results hold for the Navier \eqref{equazione},\eqref{IBC}, and the Kuttler-Sigillito \eqref{equazione},\eqref{KBC} problems.

These results are obtained by an extensive application of the averaged variational principle (AVP), that we recall here in the formulation available in \cite{EHIS}.
\begin{lem}
Consider a self-adjoint operator $H$ on a Hilbert space $\mathcal{H}$,
the spectrum of which is discrete at least in its lower portion, so that
$- \infty < \omega_0 \le \omega_1 \le \dots$.
The corresponding orthonormalized eigenvectors
are denoted $\{\mathbf{\psi}^{(j)}\}$.  The
closed quadratic form corresponding to $H$
is denoted $Q(\varphi, \varphi)$ for vectors $\varphi$ in
the quadratic-form domain $\mathcal{Q}(H) \subset \mathcal{H}$.
Let $f_p \in \mathcal{Q}(H)$ be
a family of
vectors indexed by
a variable $p$ ranging over
a measure space $(\mathfrak{M},\Sigma,\sigma)$.
Suppose that $\mathfrak{M}_0$ is a
subset of $\mathfrak{M}$.  Then for any  $z \in \mathbb{R}$,
\begin{equation}\label{RieszVersion}
	\sum_{j}{\left(z - \omega_j\right)_{+} \int_{\mathfrak{M}}\left|\langle\mathbf{\psi}^{(j)}, f_p\rangle\right|^2\,d \sigma}
       \geq
	\int_{\mathfrak{M}_0}{\left(z\| f_p\|^2 - Q(f_p,f_p) \right) d \sigma},
\end{equation}
provided that the integrals converge.
\end{lem}


\section{Comparison of eigenvalues and eigenvalue asymptotics}

In this section we provide some new results concerning the eigenvalues of problems \eqref{equazione}--\eqref{NBC}. First, we provide inequalities between eigenvalues of the problems we introduced in the previous section. Then, we complete the section by computing their asymptotics up to the second term.


\subsection{Comparison of eigenvalues}

We start with the following

\begin{thm}
\label{3.1}
Let $\Omega$ be a bounded domain in $\mathbb R^d$ with Lipschitz boundary. Then the following inequalities hold.
\begin{itemize}
\item For any $j\in\mathbb N$, and for any $ a \in(-(d-1)^{-1},1)$,
\begin{equation}
\label{dirnav}
M_j\leq\tilde M_j, \tilde\Lambda_j\leq\Lambda_j.
\end{equation}
\item For any $j\in\mathbb N$,
\begin{equation}
\label{fullchain}
\lambda_j^2\le\Lambda_j.
\end{equation}
\item If in addition $\Omega$ is convex, then for any $j\in\mathbb N$, and for any $ a \in(-(d-1)^{-1},1)$,
\begin{equation}
\label{fullchain2}
M_j(a)\le\mu_j^2.
\end{equation}
\end{itemize}
\end{thm}

We observe that all the quantities in \eqref{dirnav}--\eqref{fullchain2} have the same Weyl limit, while the respective second terms already agree with these inequalities, see Theorem \ref{asymptotal} below.

We also remark that inequality \eqref{fullchain} holds under the milder assumption that $\Omega$ is an open set of finite measure. On the other hand, if the boundary $\partial\Omega$ is assumed to be at least $C^2$, then it becomes a strict inequality. For a proof of this fact we refer to \cite[Theorem 1.1]{liu}, where the author also provides a good survey on this type of inequalities.

\begin{proof}
Inequality \eqref{dirnav} follows directly from the respective minimax characterizations.
As for inequality \eqref{fullchain}, we start with the Cauchy-Schwarz inequality
$$
\left(\int_{\Omega}|\nabla u|^2\right)^2\le\left(\int_{\Omega} u^2\right)\left(\int_{\Omega} (\Delta u)^2\right),
$$
which is valid for all $u\in H^2(\Omega)\cap H^1_0(\Omega)$.
From this, we get
$$
\left(\frac{\int_{\Omega} |\nabla u|^2}{\int_{\Omega} u^2}\right)^2\le\frac{\int_{\Omega} (\Delta u)^2}{\int_{\Omega} u^2}
$$
for all $u\in H^2(\Omega)\cap H^1_0(\Omega)$, in particular for $u\in H^2_0(\Omega)$. From this inequality, if we choose a linear, finite dimensional subspace $V\subset H^2_0(\Omega)$, we get
$$
\max_{u\in V\setminus \{0\}}\left(\frac{\int_{\Omega} |\nabla u|^2}{\int_{\Omega} u^2}\right)^2\le\max_{u\in V\setminus \{0\}}\frac{\int_{\Omega} (\Delta u)^2}{\int_{\Omega} u^2},
$$
irrespective of the choice of $V$. At this point, we may think of this as an inequality between two functions of $V$:
$$
F(V)=\max_{u\in V\setminus \{0\}}\left(\frac{\int_{\Omega} |\nabla u|^2}{\int_{\Omega} u^2}\right)^2,\ \ G(V)=\max_{u\in V\setminus \{0\}}\frac{\int_{\Omega} (\Delta u)^2}{\int_{\Omega} u^2},
$$
and 
$$
F(V)\le G(V),
$$
where $V$ varies among all the finite dimensional subspaces of $H^2_0(\Omega)$. We may as well fix a natural $j$ and restrict our attention to subspaces of dimension $j$, which is a subset of all the finite dimensional subspaces. So, it makes sense to consider the infimum, namely
$$
\inf_{\substack{V\subset H^2_0(\Omega) \\ {\rm dim\ }V=j}} F(V)\le\inf_{\substack{V\subset H^2_0(\Omega) \\ {\rm dim\ }V=j}} G(V),
$$
the inequality holding since it holds pointwise. If we now analyze both sides of the inequality, we recover that
$$
\inf_{\substack{V\subset H^2_0(\Omega) \\ {\rm dim\ }V=j}} G(V) =\min_{\substack{V\subset H^2_0(\Omega) \\ {\rm dim\ }V=j}}\max_{u\in V\setminus \{0\}}\frac{\int_{\Omega} (\Delta u)^2}{\int_{\Omega} u^2}=\Lambda_j,
$$
since the min-max is always achieved by the corresponding eigenfunctions (i.e., the infimum is achieved choosing $V$ as the space generated by the first $j$ eigenfunctions), while
$$
\inf_{\substack{V\subset H^2_0(\Omega) \\ {\rm dim\ }V=j}} F(V)=\inf_{\substack{V\subset H^2_0(\Omega) \\ {\rm dim\ }V=j}} \max_{u\in V\setminus \{0\}}\left(\frac{\int_{\Omega} |\nabla u|^2}{\int_{\Omega} u^2}\right)^2.
$$

Now note that, if we consider sets $A\subset\mathbb R_+$ (meaning that if $\alpha\in A$ then $\alpha\ge 0$), then
$$
\inf_A \alpha^2=(\inf_A \alpha)^2,\ \min_A \alpha^2=(\min_A \alpha)^2,\ \sup_A \alpha^2=(\sup_A \alpha)^2,\ \max_A \alpha^2=(\max_A \alpha)^2,
$$
since $f(x)=x^2$ is an increasing continuous function of the positive real numbers onto themselves. Hence
$$
\inf_{\substack{V\subset H^2_0(\Omega) \\ {\rm dim\ }V=j}} \max_{u\in V\setminus \{0\}}\left(\frac{\int_{\Omega} |\nabla u|^2}{\int_{\Omega} u^2}\right)^2=\left(\inf_{\substack{V\subset H^2_0(\Omega) \\ {\rm dim\ }V=j}} \max_{u\in V\setminus \{0\}}\frac{\int_{\Omega} |\nabla u|^2}{\int_{\Omega} u^2}\right)^2.
$$
The final step is increasing the space on which the infimum is taken:
$$
\inf_{\substack{V\subset H^2_0(\Omega) \\ {\rm dim\ }V=j}} \max_{u\in V\setminus \{0\}}\frac{\int_{\Omega} |\nabla u|^2}{\int_{\Omega} u^2}\ge \inf_{\substack{V\subset H^1_0(\Omega) \\ {\rm dim\ }V=j}} \max_{u\in V\setminus \{0\}}\frac{\int_{\Omega} |\nabla u|^2}{\int_{\Omega} u^2}=\lambda_j.
$$
This proves \eqref{fullchain}.

Regarding \eqref{fullchain2}, we assume first that $\Omega$ is smooth ($C^{\infty})$. We note that for any smooth function $u$ on $\Omega$ we have
\begin{equation}\label{part}
\int_{\Omega}|D^2u|^2dx=\int_{\Omega}(\Delta u)^2dx+\frac{1}{2}\int_{\partial\Omega}\frac{\partial\ }{\partial\nu}(|\nabla u|^2)d\sigma-\int_{\partial\Omega}\Delta u\frac{\partial u}{\partial\nu}d\sigma,
\end{equation}
where $d\sigma$ the measure element of $\partial\Omega$. Equality \eqref{part} follows from the pointwise identity $|D^2u|^2=\frac{1}{2}\Delta(|\nabla u|^2)-\nabla\Delta u\cdot\nabla u$. Now we note that, on $\partial\Omega$,
\begin{equation}\label{part2}
\frac{1}{2}\frac{\partial\ }{\partial\nu}|\nabla u|^2=\nabla \frac{\partial u}{\partial\nu}\cdot\nabla u-\nabla u^T\cdot D\nu\cdot\nabla u
=\nabla_{\partial\Omega}\frac{\partial u }{\partial\nu}\cdot\nabla_{\partial\Omega} u+\frac{\partial^2 u }{\partial\nu^2}\frac{\partial u }{\partial\nu}-II(\nabla_{\partial\Omega} u,\nabla_{\partial\Omega} u).
\end{equation}
Here $II(\cdot,\cdot)$ denotes the second fundamental form on $\partial\Omega$ (in fact $II=D\nu$). The quadratic form $II(\cdot,\cdot)$ defined on the tangent space to $\partial\Omega$ is symmetric and its
eigenvalues are the principal curvatures of $\partial\Omega$.

Assume now that $u$ is such that $\frac{\partial u }{\partial\nu}=0$ on $\partial\Omega$ and that $II\geq 0$ in the sense of quadratic forms (this holds e.g., for smooth convex domains). Then $\nabla u=\nabla_{\partial\Omega}u$ on $\partial\Omega$ (the gradient of $u$ restricted on the boundary belongs to the tangent space to the boundary). This fact combined with \eqref{part} and \eqref{part2}  implies that for such $u$ and $\Omega$
$$
\int_{\Omega}|D^2u|^2dx\leq\int_{\Omega}(\Delta u)^2dx.
$$
Since we assumed $\Omega$ smooth, then all eigenfunctions of the Neumann Laplacian belong to $C^{\infty}(\Omega)\cap H^2(\Omega)$ and satisfy $\frac{\partial u }{\partial\nu}=0$ on $\partial\Omega$. Hence, taking the space generated by the first $j$ eigenfunctions of the Neumann Laplacian as $j$-dimensional subspace of $H^2(\Omega)$ of test functions into the min-max formula for $M_j(a)$, we obtain \eqref{fullchain2}, which is then proved in the case of a smooth convex domain. Since any convex domain can be approximated uniformly by smooth convex domains, we have pointwise convergence of eigenvalues (see e.g., \cite{arrieta_lamberti_CR}). Therefore, we deduce the validity of  \eqref{fullchain2} for any convex set.
\end{proof}


\subsection{Semiclassical asymptotics}

In this section, the domain $\Omega\subset\mathbb R^d$ will always be a bounded domain, smooth enough in order to apply the arguments in \cite{safvas, vas} (see Theorem \ref{thmb}). In particular, smooth convex sets and piecewise smooth domains with non positive conormal curvature (such as polyhedra) are admissible. Moreover, the dimension $d$ will always be such that $d\ge2$. 

We parametrize $\Omega$ locally in such a way that $\Omega=\{(x_1,\dots,x_d):x_d>0\}$ and $\partial\Omega=\{(x_1,\dots,x_{d-1},0)\}$. We also denote by $(x,\xi)$ the elements of the cotangent bundle $T^*\Omega$, $\xi=(\xi_1,\dots,\xi_d)$ being the coordinates on the fiber $T^*_x\Omega$. Setting $x'=(x_1,\dots,x_{d-1})$ and $\xi'=(\xi_1,\dots,\xi_{d-1})$, we have that $(x',\xi')$ are coordinates for the cotangent bundle $T^*\partial\Omega$.

The operator $\Delta^2$ is represented by the symbol
$$
A(\xi)=|\xi|^4=\left(\sum_{k=1}^d\xi_k^2\right)^2,
$$
and the operator can be recovered from the symbol by substituting $\xi_k$ with $D_k=-i\frac{\partial\ }{\partial x_k}$. Note that this operator coincides with its principal part, i.e., the symbol only contains monomials of the same degree.

Regarding the boundary operators, we first recall that, because of the parametrization we have chosen, in this case the normal derivative is
$$
\frac{\partial\ }{\partial \nu}=-\frac{\partial\ }{\partial x_k}\ \text{(on $\partial\Omega$)}.
$$
Let us now discuss the various boundary conditions one by one.
\begin{itemize}
\item $B_0(D)u=u$. Its symbol is $B_0(\xi)=1$.
\item $B_1(D)u=\frac{\partial u}{\partial\nu}$. Its symbol is $B_1(\xi)=-i\xi_d$.
\item $B_2(D)u=(1- a )\frac{\partial^2 u}{\partial\nu^2}+ a \Delta u$. Its symbol is $B_2(\xi)=-\xi_d^2-i a  K \xi_d- a |\xi'|^2$ (where $K$ is the sum of the principal curvatures). Note that its principal part is $\tilde B_2(\xi)=-\xi_d^2- a |\xi'|^2$.
\item $B_3(D)u=\frac{\partial\Delta u}{\partial \nu}+(1- a ){\rm div}_{\partial\Omega}\left(\frac{\partial\ }{\partial \nu}\nabla_{\partial\Omega}u\right)$. Writing the symbol for this operator is quite complicated, but using the equality
$$
{\rm div}_{\partial\Omega}\left(\frac{\partial\ }{\partial \nu}\nabla_{\partial\Omega}u\right)=\Delta_{\partial\Omega} \frac{\partial u}{\partial \nu}-{\rm div}_{\partial\nu}(\nabla_{\partial\Omega}u\cdot D \nu)
$$
we can easily write the principal part $\tilde B_3(\xi)=i \xi_d^3+i(2- a )\xi_d|\xi'|^2$.
\end{itemize}

Now we introduce an auxiliary problem related with problems \eqref{equazione}--\eqref{NBC}:
\begin{equation}
\label{auxiliary}
\left\{
\begin{array}{ll}
A(\xi',D_d)v(x_d)=\eta v(x_d),& x_d\in(0,+\infty),\\
\tilde B_j(\xi',D_d)v|_{x_d=0}=0, &\ \ 
\end{array} \right.
\end{equation}
where the boundary conditions will be: $j=0,1$ for the Dirichlet case,
$$
v(0)=v'(0)=0,
$$
$j=0,2$ for the Navier case,
$$
v(0)=v''(0)=0,
$$
$j=1,3$ for the Kuttler-Sigillito case,
$$
v'(0)=v'''(0)=0,
$$
or $j=2,3$ for the Neumann case,
$$
v''(0)- a |\xi'|^2v(0)=v'''(0)-(2- a )|\xi'|^2v'(0)=0.
$$
Note that problem \eqref{auxiliary} depends on $\xi'\in\mathbb R^{d-1}$.

We are interested in the spectrum of problem \eqref{auxiliary}. We start by observing that there are no eigenvalues, with the sole exception of the Neumann case with $a\neq 0$, where there is a simple eigenvalue
\begin{equation*}
\eta=\eta(\xi')=f( a )|\xi'|^4,
\end{equation*}
where
\begin{equation}
\label{fsigma}
f( a )=4 a -1-3 a ^2+2(1- a )\sqrt{2 a ^2-2 a +1}.
\end{equation}
Notice that $0<f( a )\le1$  for $ a \in(-(d-1)^{-1},1)$, with $f( a )=1$ only for $ a =0$. We remark that the case $ a =0$ does not have eigenvalues, hence neither is $|\xi'|^4$ (differently from the case $a\neq 0$).

In addition, problem \eqref{auxiliary} is known to have as essential spectrum the strip $[|\xi'|^4,+\infty[$ (see e.g., \cite[Appendix A]{safvas}). Moreover, the essential spectrum has only one threshold with one double root. A threshold $\eta^{st}$ is a point in the essential spectrum for which the equation
$$
A(\xi', \zeta)=\eta^{st}
$$
has a multiple real root. It is clear that, in our case, the only threshold is $\eta^{st}=|\xi'|^4$. At this point we search for generalized eigenfunctions in the strip $]\eta^{st},+\infty[$. To do so, we have first to solve the equation
\begin{equation}
\label{genev}
A(\xi',\zeta)=\eta,
\end{equation}
for any $\eta\in]\eta^{st},+\infty[$. Equation \eqref{genev} always has four roots:
\begin{equation*}
\zeta_1^-=-\sqrt{\sqrt{\eta}-|\xi'|^2},\ \zeta_1^+=\sqrt{\sqrt{\eta}-|\xi'|^2},\ \zeta_2^-=-i\sqrt{\sqrt{\eta}+|\xi'|^2},\ \zeta_2^+=i\sqrt{\sqrt{\eta}+|\xi'|^2}.
\end{equation*}
We then search for generalized eigenfunctions (associated with $\eta$) of the form
\begin{equation}
\label{genef}
v(x_d)=a_1^-e^{i\zeta_1^-x_d}+a_1^+e^{i\zeta_1^+x_d}+a_2^+e^{i\zeta_2^+x_d}.
\end{equation}
Note that these generalized eigenfunctions are not proper eigenfunctions (because they are not $L^2$-functions), nevertheless they are bounded solutions. We search for generalized eigenfuntions because we need to compute the quantity ${\rm arg}\left(i\frac{a_1^+}{a_1^-}\right)$, where ${\rm arg}$ is the standard complex argument of a number.

\begin{itemize}

\item {\bf Dirichlet problem.} Through the boundary conditions we get
$$
\left\{\begin{array}{l}
a_1^-+a_1^++a_2^+=0,\\
\zeta_1^-a_1^-+\zeta_1^+a_1^++\zeta_2^+a_2^+=0,
\end{array}\right.
$$
hence
$$
\frac{a_1^+}{a_1^-}=-\frac{\zeta_1^--\zeta_2^+}{\zeta_1^+-\zeta_2^+}=-\frac{|\xi'|^2}{\sqrt{\eta}}+i\frac{\sqrt{\eta-|\xi'|^4}}{\sqrt{\eta}},
$$
from which we obtain
\begin{equation}
\label{argdir}
{\rm arg}\left(i\frac{a_1^+}{a_1^-}\right)=\arctan{\frac{|\xi'|^2}{\sqrt{\eta-|\xi'|^4}}}-\pi+2k\pi=\arcsin{\frac{|\xi'|^2}{\sqrt{\eta}}}-\pi+2k\pi,
\end{equation}
for some $k\in\mathbb Z$.

\item {\bf Navier problem.} Through the boundary conditions we get
$$
\left\{\begin{array}{l}
a_1^-+a_1^++a_2^+=0,\\
(\zeta_1^-)^2a_1^-+(\zeta_1^+)^2a_1^++(\zeta_2^+)^2a_2^+=0,
\end{array}\right.
$$
that yields $a_1^+/a_1^-=-1$, hence
\begin{equation}
\label{argnav}
{\rm arg\ }\left(i\frac{a_1^+}{a_1^-}\right)=-\frac \pi 2+2k\pi,
\end{equation}
for some $k\in\mathbb Z$.

\item {\bf Kuttler-Sigillito problem.} Through the boundary conditions we get
$$
\left\{\begin{array}{l}
\zeta_1^-a_1^-+\zeta_1^+a_1^++\zeta_2^+a_2^+=0, \\
(\zeta_1^-)^3a_1^-+(\zeta_1^+)^3a_1^++(\zeta_2^+)^3a_2^+=0,
\end{array}\right.
$$
that yields $a_1^+/a_1^-=1$, hence
\begin{equation}
\label{argks}
{\rm arg\ }\left(i\frac{a_1^+}{a_1^-}\right)=\frac \pi 2+2k\pi,
\end{equation}
for some $k\in\mathbb Z$.

\item {\bf Neumann problem.} Through the boundary conditions we get
$$
\left\{\begin{array}{l}
-(\zeta_1^-)^2a_1^--(\zeta_1^+)^2a_1^+-(\zeta_2^+)^2a_2^+- a |\xi'|^2(a_1^-+a_1^++a_2^+)=0,\\
-i(\zeta_1^-)^3a_1^--i(\zeta_1^+)^3a_1^+-i(\zeta_2^+)^3a_2^+-i(2- a )|\xi'|^2(\zeta_1^-a_1^-+\zeta_1^+a_1^++\zeta_2^+a_2^+)=0,
\end{array}\right.
$$
that yields
\begin{multline*}
\zeta_2^+\left((\zeta_1^+)^2+ a |\xi'|^2\right)\left((\zeta_2^+)^2+(2- a )|\xi'|^2\right)\left(\frac{a_1^+}{a_1^-}+1\right)\\
=\zeta_1^+\left((\zeta_2^+)^2+ a |\xi'|^2\right)\left((\zeta_1^+)^2+(2- a )|\xi'|^2\right)\left(\frac{a_1^+}{a_1^-}-1\right).
\end{multline*}
Therefore
$$
\frac{a_1^+}{a_1^-}=\frac{A+iB}{A-iB}=\frac{(A+iB)^2}{A^2+B^2},
$$
where
$$
A=\sqrt{\sqrt{\eta}-|\xi'|^2}\left(\sqrt{\eta}+(1- a )|\xi'|^2\right)^2,\\
B=\sqrt{\sqrt{\eta}+|\xi'|^2}\left(\sqrt{\eta}-(1- a )|\xi'|^2\right)^2.
$$
In particular
\begin{equation}
\label{argneu}
{\rm arg\ }\left(i\frac{a_1^+}{a_1^-}\right)={\rm arg\ }(i)+2{\rm arg\ }(A+iB)=-\frac\pi2-2\arctan{\frac A B}+2k\pi,
\end{equation}
for some $k\in\mathbb Z$.
\end{itemize}


We now recall the following theorem (\cite[Theorem 1.6.1]{safvas}).

\begin{thm}
Let $\Omega$ be piecewise $C^{\infty}$ and satisfying the nonperiodicity and nonblocking conditions, and let $d\geq 2$. Let $N(z)$ be the counting function associated with the biharmonic operator with either Dirichlet, Navier, or Neumann boundary conditions, i.e., the problem given by equation \eqref{equazione} coupled with \eqref{DBC}, \eqref{IBC}, or \eqref{NBC}, respectively.
Then, for $z\to+\infty$ we have
\begin{equation*}
N(z)=c_0 z^{\frac d 4}+c_1 z^{\frac {d-1} 4}+o\left(z^{\frac {d-1} 4}\right),
\end{equation*}
where
$$
c_0=(2\pi)^{-d}\int_{|\xi|^4\le1}dxd\xi,\ \ 
c_1=(2\pi)^{1-d}\int_{T^*\partial\Omega}{\rm shift}^+(1,\xi')dx'd\xi'.
$$
Here ${\rm shift}^+$ is the shift function associated with problem \eqref{auxiliary}, and there exists an analytic branch ${\rm arg}_0$ of the argument ${\rm arg}$ such that we have
$$
{\rm shift}^+(\eta,\xi')=N^+(\eta,\xi')+\frac{{\rm arg}_0{\rm det}(iR(\eta,\xi'))}{2\pi},
$$
where $N^+$ is the counting function of problem \eqref{auxiliary}, and $R$ is the reflexion matrix associated with problem \eqref{auxiliary}, in particular
$$
{\rm det}(iR(\eta,\xi'))=
\left\{\begin{array}{ll}
0, & {\rm if\ }\eta\le\eta^{st},\\
i\frac{a_1^+}{a_1^-}, & {\rm otherwise},
\end{array}\right.
$$
with $a_1^{\pm}$ defined in \eqref{genef}.

In addition, the function ${\rm arg}_0$ is a suitable branch of the complex argument satisfying the following condition
$$
\lim_{\eta\to|\xi'|^4}\left|{\rm arg}_0\left(i\frac{a_1^+}{a_1^-}\right)\right|=\frac\pi 2.
$$
\end{thm}

We stress the fact that the function ${\rm arg}_0$ depends on the particular problem that is considered, and not a function chosen once and for all.

\begin{cor}
Let $\Omega$ be piecewise $C^{\infty}$ and satisfying the nonperiodicity and nonblocking conditions, and let $d\geq 2$. If $\omega_j$ is the $j$-th eigenvalue of the biharmonic operator with either Dirichlet, Navier, or Neumann boundary conditions, then we have
\begin{equation*}
\omega_j=\left(\frac{j}{c_0}\right)^{\frac 4 d}-\frac{4c_1}{dc_0^{\frac {d+3}d}}j^{\frac 3 d}+o\left(j^{\frac 3 d}\right),
\end{equation*}
or equivalently
\begin{equation*}
\omega_j^{\frac14}=\left(\frac{j}{c_0}\right)^{\frac 1 d}-\frac{c_1}{dc_0}+o(1),
\end{equation*}
as $j\to+\infty$.
\end{cor}

Now we compute the coefficients $c_0,c_1$. As for $c_0$, it depends only on the equation and therefore will be the same for both Dirichlet, Navier, and Neumann boundary conditions, and it is
\begin{equation*}
c_0=(2\pi)^{-d}\int_{|\xi|^4\le1}\int_{\Omega}dxd\xi=(2\pi)^{-d}B_d|\Omega|.
\end{equation*}
As for $c_1$, its definition sensitively depends on the boundary conditions, so we split the discussion.
\begin{itemize}

\item {\bf Dirichlet boundary conditions.} We have seen that problem \eqref{auxiliary} has no eigenvalue, and it is easy to check that the function ${\rm arg}_0$ is given by formula \eqref{argdir} with $k=0$. Hence
\begin{equation}
\label{c1dir}
\begin{split}
c_1 & =(2\pi)^{-d}|\partial\Omega|\left(\int_{|\xi'|<1}\arcsin{|\xi'|^2}d\xi'-\pi B_{d-1}\right)\\
 & =-\frac{ B_{d-1}|\partial\Omega|}{4(2\pi)^{d-1}}\left(1+\frac{\Gamma\left(\frac{d+1}{4}\right)}{\sqrt{\pi}\Gamma\left(\frac{d+3}{4}\right)}\right).
\end{split}
\end{equation}

\item {\bf Navier boundary conditions.} We have seen that problem \eqref{auxiliary} has no eigenvalue, and it is easy to check that the function ${\rm arg}_0$ is given by formula \eqref{argnav} with $k=0$. Hence
\begin{equation}
c_1=(2\pi)^{1-d}|\partial\Omega|\int_{|\xi'|<1}\left(-\frac 1 4\right)d\xi'=-\frac{ B_{d-1}|\partial\Omega|}{4(2\pi)^{d-1}}.
\end{equation}

\item {\bf Kuttler-Sigillito boundary conditions.} We have seen that problem \eqref{auxiliary} has no eigenvalue, and it is easy to check that the function ${\rm arg}_0$ is given by formula \eqref{argks} with $k=0$. Hence
\begin{equation}
c_1=(2\pi)^{1-d}|\partial\Omega|\int_{|\xi'|<1}\frac 1 4 d\xi'=\frac{ B_{d-1}|\partial\Omega|}{4(2\pi)^{d-1}}.
\end{equation}

\item {\bf Neumann boundary conditions.} Let us start with the case $ a \neq0$. Here we have seen that problem \eqref{auxiliary} has a simple eigenvalue
$$
\eta=f( a )|\xi'|^4,
$$
so that 
$$
N^+(\lambda,\xi')=
\left\{\begin{array}{ll}
1,& {\rm if\ }|\xi'|<f( a )^{-\frac 1 4},\\
0, & {\rm otherwise}.
\end{array}\right.
$$
It is also easily checked that the function ${\rm arg}_0$ is given by formula \eqref{argneu} with $k=0$, therefore
\begin{equation}
\label{c1neu}
c_1=\frac{ B_{d-1}|\partial\Omega|}{4(2\pi)^{d-1}}\left(4f( a )^{\frac{1-d}{4}}-1-4\frac{d-1}{\pi}\int_0^1t^{d-2}\arctan g(t, a )dt\right),
\end{equation}
where
\begin{equation}
\label{gneu}
g(t, a )=\frac{\sqrt{1-t^2}\left(1+(1- a )t^2\right)^2}{\sqrt{1+t^2}\left(1-(1- a )t^2\right)^2}.
\end{equation}

If instead we consider the case $ a =0$, we recall that there are no eigenvalues, however now the function ${\rm arg}_0$ is given by formula \eqref{argneu} but with $k=1$, so that here
$$
c_1=\frac{ B_{d-1}|\partial\Omega|}{4(2\pi)^{d-1}}\left(3-4\frac{d-1}{\pi}\int_0^1t^{d-2}\arctan g(t,0)dt\right),
$$
and in particular, as $f(0)=1$, we have that formula \eqref{c1neu} still holds.

We observe that, by using the equality
$$
\arctan x +\arctan \frac 1 x =\frac \pi 2,\ \ \forall x>0,
$$
as $g(t, a )>0$ for all $t\in(0,1)$ and for all $ a $, we obtain the equivalent formula
\begin{equation*}
c_1=\frac{ B_{d-1}|\partial\Omega|}{4(2\pi)^{d-1}}\left(4f( a )^{\frac{1-d}{4}}-3+4\frac{d-1}{\pi}\int_0^1t^{d-2}\arctan (g(t, a )^{-1})dt\right).
\end{equation*}
\end{itemize}

Summing up, we have the following

\begin{thm}
\label{asymptotal}
Let $\Omega$ be piecewise $C^{\infty}$ and satisfying the nonperiodicity and nonblocking conditions, and let $d\geq 2$. Let $C_d=(2\pi)^2B_d^{-\frac 2 d}$. For any $a\in(-(d-1)^{-1},1)$, the following expansions hold:
\begin{equation}\label{weyl_dirichlet_biharmonic_single}
\Lambda_j=C_d^2\left(\frac{j}{|\Omega|}\right)^{\frac{4}{d}}+\frac{C_d^2 B_{d-1}}{d B_d^{1-\frac{1}{d}}}\left(1+\frac{\Gamma\left(\frac{d+1}{4}\right)}{\sqrt{\pi}\Gamma\left(\frac{d+3}{4}\right)}\right)\frac{|\partial\Omega|}{|\Omega|}\left(\frac{j}{|\Omega|}\right)^{\frac{3}{d}}+o\left(j^{\frac{3}{d}}\right),
\end{equation}

\begin{equation}\label{weyl_intermediate_biharmonic_single}
\tilde\Lambda_j(a)=C_d^2\left(\frac{j}{|\Omega|}\right)^{\frac{4}{d}}+\frac{C_d^2 B_{d-1}}{d B_d^{1-\frac{1}{d}}}\frac{|\partial\Omega|}{|\Omega|}\left(\frac{j}{|\Omega|}\right)^{\frac{3}{d}}+o\left(j^{\frac{3}{d}}\right),
\end{equation}

\begin{equation}\label{weyl_KS_biharmonic_single}
\tilde M_j(a)=C_d^2\left(\frac{j}{|\Omega|}\right)^{\frac{4}{d}}-\frac{C_d^2 B_{d-1}}{d B_d^{1-\frac{1}{d}}}\frac{|\partial\Omega|}{|\Omega|}\left(\frac{j}{|\Omega|}\right)^{\frac{3}{d}}+o\left(j^{\frac{3}{d}}\right),
\end{equation}
and

\begin{multline*}
M_j(a)=C_d^2\left(\frac{j}{|\Omega|}\right)^{\frac{4}{d}}\\
-\frac{C_d^2 B_{d-1}}{d B_d^{1-\frac{1}{d}}}\left(4f(a)^{\frac{1-d}{4}}-1-4\frac{d-1}{\pi}\int_0^1t^{d-2}\arctan{g(t,a)}dt\right)\frac{|\partial\Omega|}{|\Omega|}\left(\frac{j}{|\Omega|}\right)^{\frac{3}{d}}+o\left(j^{\frac{3}{d}}\right),
\end{multline*}
as $j\to\infty$, for any $a\in(-(d-1)^{-1},1)$, where $f$ is defined in \eqref{fsigma} and $g$ is defined in \eqref{gneu}.
\end{thm}

We conclude this discussion with a few remarks.

\begin{rem}
It is interesting to see that, contrary to what happens with the Laplacian, in the case of the biharmonic operator the quantity $|c_1|$ is not the same for Dirichlet and Neumann eigenvalues. In fact, this is the case even for $ a =0$. In addition, the dependence on the dimension is even stronger, and it is actually worth noticing that, as the dimension grows, the asymptotics of (the square root) of the eigenvalues of the Dirichlet Bilaplacian converge to that of the eigenvalues of the Dirichlet Laplacian, because
$$
\lim_{d\to\infty}1+\frac{\Gamma\left(\frac{d+1}{4}\right)}{\sqrt{\pi}\Gamma\left(\frac{d+3}{4}\right)}=1,
$$
and hence the inequality
$$
\lambda_j\le\sqrt{\Lambda_j}
$$
is, in a sense, ``squeezing'' towards an equality, asymptotically in $j$ and in $d$.
On the other hand, the Dominated Convergence Theorem tells us that
$$
\lim_{d\to\infty}4\frac{d-1}{\pi}\int_0^1t^{d-2}\arctan (g(t, a )^{-1})dt=0
$$
for all $ a \in((d-1)^{-1},1)$. However,
$$
\lim_{d\to\infty}f( a )^{\frac{1-d}{4}}=\left\{\begin{array}{ll}1,& a =0,\\+\infty,&\text{otherwise,}\end{array}\right.
$$
telling us that the asymptotics of (the square root) of the Neumann Bilaplacian eigenvalues converge to that of the Neumann Laplacian eigenvalues only for $ a =0$, while in the other cases the asymptotic expansions blow up. This can be interpreted as the fact that, when the dimension increases, the control of the Hessian matrix on the Laplacian (expressed by the Poisson ratio in the quadratic form \eqref{qf}) weakens significantly, making the asymptotics blow up.
\end{rem}

\begin{rem}
We observe that, if $\Omega$ satisfies the uniform outer ball condition (see \cite{adolfsson, ggs}), then the expansion \eqref{weyl_intermediate_biharmonic_single} holds also for $a=1$. The same remark applies to \eqref{weyl_KS_biharmonic_single}. On the other hand, even though the Neumann problem \eqref{equazione}, \eqref{NBC} does not satisfy the complementing condition (see \cite{ggs}) when $ a =1$ and the operator does not have compact resolvent, and therefore all the discussion in this section does not apply, it is nevertheless interesting to see what happens to $c_1$ as $ a \to1-$. We observe that
$$
\left.-1-4\frac{d-1}{\pi}\int_0^1t^{d-2}\arctan g(t, a )dt\right|_{ a =1}=-1-\frac{\Gamma\left(\frac{d+1}{4}\right)}{\sqrt{\pi}\Gamma\left(\frac{d+3}{4}\right)},
$$
while
$$
\lim_{ a \to1-}f( a )^{\frac{1-d}4}=+\infty.
$$
This is coherent with what we know about the spectrum of this operator: apart from an infinite dimensional kernel, the remaining part of the spectrum consists of the eigenvalues of the Dirichlet Bilaplacian, see \cite{provneu}.
\end{rem}

\begin{rem}
It is striking that the asymptotics for the Navier and the Kuttler-Sigillito problems are the same as those of the Dirichlet Laplacian and the Neumann Laplacian, respectively. In particular, the dependence on the Poisson ratio is not visible in those expansions, and the link with the respective Laplacian counterpart becomes evident. However, apart from the case $a=1$ where the identification becomes immediate, it is not clear at all what are the relations between the respective Laplacian and Bilaplacian eigenvalues.
\end{rem}


\section{The biharmonic Dirichlet operator}

In this section we focus our attention to the biharmonic Dirichlet problem \eqref{equazione}, \eqref{DBC}. In particular, the quadratic form \eqref{qf} will be set into $H^2_0(\Omega)$. Through all this section, $\Omega\subset \mathbb R^d$ will be a domain with finite Lebesgue measure, unless otherwise specified. In fact, since the embedding $H^1_0(\Omega)\subset L^2(\Omega)$ is compact under the sole assumption that the measure of $\Omega$ is finite, it is standard to see that the spectrum is discrete and consists of an ordered sequence of positive eigenvalues tending to infinity.

 Note that the quadratic form \eqref{qf} is now equal to
\begin{equation*}
  Q(u,v)=\int_{\Omega}D^2u:D^2v=\int_{\Omega}\Delta u \Delta v,
\end{equation*}
so the dependence upon the Poisson ratio disappears.

We also observe here that, directly from \eqref{weyl_dirichlet_biharmonic_single}, we have the following asymptotic law for averages of eigenvalues, holding if we require suitable assumptions on $\Omega$ (see Theorem \ref{asymptotal})
\begin{equation}\label{weyl_dirichlet_biharmonic}
\frac{1}{k}\sum_{j=1}^k\Lambda_j=
\frac{d}{d+4}C_d^2\left(\frac{k}{|\Omega|}\right)^{\frac{4}{d}}
+\frac{d}{d+3}\frac{C_d^2 B_{d-1}}{d B_d^{1-\frac{1}{d}}}\left(1+\frac{\Gamma\left(\frac{d+1}{4}\right)}{\sqrt{\pi}\Gamma\left(\frac{d+3}{4}\right)}\right)\frac{|\partial\Omega|}{|\Omega|}\left(\frac{k}{|\Omega|}\right)^{\frac{3}{d}}
+o\left(k^{\frac{3}{d}}\right)
\end{equation}
as $k\rightarrow+\infty$, where $C_d=(2\pi)^2B_d^{-\frac 2 d}$.


\subsection{Lower bounds for Riesz means}

In this section we will apply the averaged variational principle to obtain lower bounds for Riesz means (respectively, upper bounds for averages) of eigenvalues of $\Delta^2_D$ on $\Omega$. 

Applying the AVP \eqref{RieszVersion} with test functions of the form $f_{ p }(x)=(2\pi)^{-d/2}e^{i p \cdot x}\phi(x)$, with $\phi(x)\in H^2_0(\Omega)\cap L^{\infty}(\Omega)$, we obtain the following

\begin{thm}\label{dirichlet_bilaplacian_thm_general}
Let $\Omega$ be a domain in $\mathbb R^d$ of finite measure. For any $\phi\in H^2_0(\Omega)\cap L^{\infty}(\Omega)$ and $z>0$ the following inequality holds
\begin{multline}\label{Riesz-mean-ineq-DirichletbiLaplacian}
\sum_{j\ge 1}\left(z-\Lambda_j\right)_+\|\phi U_j\|_2^2\geq \frac{4}{d+4}(2\pi)^{-d} B_d\|\phi\|_2^2\left(z-\frac{\|\Delta\phi\|_2^2}{\|\phi\|_2^2}\right)^{\frac{d}{4}+1}_+\\
-2(2\pi)^{-d} B_d\|\nabla\phi\|_2^2\left(z-\frac{\|\Delta\phi\|_2^2}{\|\phi\|_2^2}\right)^{\frac{d}{4}+\frac{1}{2}}_+.
\end{multline}
Moreover, for all positive integers $k$
\begin{equation}\label{evsums-DirichletbiLaplacian1}
 \frac{1}{k}\sum_{j=1}^k\Lambda_j\leq\frac{d}{d+4}C_d^2\left(\frac{k}{|\Omega|}\right)^{\frac{4}{d}}\rho(\phi)^{-\frac{4}{d}}
+2\frac{\|\nabla\phi\|_2^2}{\|\phi\|_2^2}C_d\left(\frac{k}{|\Omega|}\right)^{\frac{2}{d}}\rho(\phi)^{-\frac{2}{d}}
+\frac{\|\Delta\phi\|_2^2}{\|\phi\|_2^2},
\end{equation}
for $\rho(\phi)<1$, where
\begin{equation}
\label{rhophi}
\rho(\phi)=\frac{||\phi||_2^2}{|\Omega|\cdot||\phi||_{\infty}^2}.
\end{equation}
\end{thm}

\begin{proof}
We take in \eqref{RieszVersion} trial functions of the form $f_{ p }=(2\pi)^{-d/2}e^{i p \cdot x}\phi(x)$ with $\phi\in H^2_0(\Omega)\cap L^{\infty}(\Omega)$ real valued. After averaging over $ p \in\mathbb R^d$ and using the unitarity of the Fourier transform we get, for any $R>0$,
\begin{equation}\label{first_step}
\sum_{j\ge 1}\left(z-\Lambda_j\right)_+\int_{\Omega}\phi^2(x)U_j^2(x)dx
\geq(2\pi)^{-d}\int_{| p |\leq R}\left(z\|\phi\|_2^2-\int_{\Omega}\left|\Delta(\phi e^{i p \cdot x})\right|^2dx\right)d p.
\end{equation}

Now we note that
\begin{equation*}
\begin{split}
\int_{\Omega}\left|\Delta(\phi e^{i p \cdot x})\right|^2 & =\int_{\Omega}\left|\Delta\phi-| p |^2\phi+2i p\cdot\nabla\phi\right|^2dx\\
& =\int_{\Omega}(\Delta\phi-| p |^2\phi)^2+4\left| p \cdot\nabla\phi\right|^2dx\\
& =\int_{\Omega}(\Delta\phi)^2+| p |^4\phi^2-2| p |^2\phi\Delta\phi+4\left| p \cdot\nabla\phi\right|^2dx\\
& =\int_{\Omega}(\Delta\phi)^2+| p |^4\phi^2+2| p |^2|\nabla\phi|^2+4\left| p \cdot\nabla\phi\right|^2dx,
\end{split}
\end{equation*}
which implies
\begin{multline}\label{step_2}
\sum_{j\ge 1}\left(z-\Lambda_j\right)_+\int_{\Omega}\phi^2(x)U_j^2(x)dx\\
\geq(2\pi)^{-d}\int_{| p |\leq R}\left((z-| p |^4)\|\phi\|_2^2-2| p |^2\|\nabla\phi\|_2^2-\|\Delta\phi\|_2^2-4\int_\Omega|p\cdot\nabla\phi|^2dx^2\right)d p \\
=(2\pi)^{-d} B_d\|\phi\|_2^2\left(\left(z-\frac{\|\Delta\phi\|_2^2}{\|\phi\|_2^2}\right)R^d-\frac{d}{d+4}R^{d+4}-2\frac{\|\nabla\phi\|_2^2}{\|\phi\|_2^2}R^{d+2}\right).
\end{multline}
Choosing
$$
R^4=\left(z-\frac{\|\Delta\phi\|_2^2}{\|\phi\|_2^2}\right)_+,
$$
we obtain \eqref{Riesz-mean-ineq-DirichletbiLaplacian}.

Now we can consider \eqref{first_step} with the evaluation $z=\Lambda_{k+1}$, so that the sum at the left-hand side is taken over the first $k$ positive integers. Hence, as in \eqref{step_2}, we get 
\begin{multline}\label{step_3}
\|\phi\|_{\infty}^2\sum_{j=1}^k\left(\Lambda_{k+1}-\Lambda_j\right)\geq \sum_{j=1}^k\left(\Lambda_{k+1}-\Lambda_j\right)\int_{\Omega}\phi^2(x)U_j^2(x)dx\\
\geq(2\pi)^{-d}\int_{| p |\leq R}\left((\Lambda_{k+1}-| p |^4)\|\phi\|_2^2-2| p |^2\|\nabla\phi\|_2^2-\|\Delta\phi\|_2^2-4\int_\Omega|p\cdot\nabla\phi|^2dx\right)d p \\
=(2\pi)^{-d} B_d\|\phi\|_2^2\left(\left(\Lambda_{k+1}-\frac{\|\Delta\phi\|_2^2}{\|\phi\|_2^2}\right)R^d-\frac{d}{d+4}R^{d+4}-2\frac{\|\nabla\phi\|_2^2}{\|\phi\|_2^2}R^{d+2}\right),
\end{multline}
for any $R>0$, where the first inequality follows from $\int_{\Omega}\phi^2(x)U_j^2(x)dx\leq\|\phi\|_{\infty}^2$. We choose now
$$
R^4=C_d^2\left(\frac{k}{|\Omega|}\right)^{\frac{4}{d}}\rho(\phi)^{-\frac{4}{d}}.
$$
Standard computations show that with this choice inequality \eqref{step_3} implies \eqref{evsums-DirichletbiLaplacian1}. 
\end{proof}

\begin{rem}
The right side of inequality \eqref{evsums-DirichletbiLaplacian1} provides a good relation between the upper bound and the semiclassical behaviour of the average of the first $k$ eigenvalues, which is known to be a lower bound for the average, see \cite{Lap1997} (see also \cite{Ber,LiYau}).
\end{rem}

As a  corollary, we have a lower bound for the partition function (the trace of the heat kernel). 

\begin{cor}\label{dirichlet_bounds_spec_fcn_bi}
For any $\phi\in H^2_0(\Omega)\cap L^{\infty}(\Omega)$ and $t>0$,
\begin{multline}\label{Partition-function-inequality_bi}
  \sum_{j=1}^{\infty}e^{-\Lambda_{j}t}\|\phi U_j\|_2^2\geq \frac{4}{d+4}(2\pi)^{-d} B_d\Gamma\left(2+\frac{d}{4}\right)\|\phi\|_2^2e^{-\frac{||\Delta\phi||_2^2}{||\phi||_2^2}\,t}t^{-\frac{d}{4}}\\-2(2\pi)^{-d} B_d\Gamma\left(\frac{3}{2}+\frac{d}{4}\right)\|\nabla\phi\|_2^2e^{-\frac{||\Delta\phi||_2^2}{||\phi||_2^2}\,t}t^{\frac{1}{2}-\frac{d}{4}}.
\end{multline}
Moreover,
\begin{equation}\label{part-fct-estimate-small-times_bi}
\begin{split}
  \sum_{j=1}^{\infty}e^{-\Lambda_{j}t}\geq  & \frac{4}{d+4}(2\pi)^{-d} B_d\Gamma\left(2+\frac{d}{4}\right)t^{-\frac{d}{4}}|\Omega|\\
	& \quad -\frac{4}{d+4}(2\pi)^{-d} B_d\Gamma\left(2+\frac{d}{4}\right)t^{-\frac{d}{4}}\left(\frac{t\|\Delta\phi\|_2^2+|\Omega|\|\phi\|_{\infty}^2-\|\phi\|_2^2}{\|\phi\|_{\infty}^2}\right)\\
	& \quad -2(2\pi)^{-d} B_d\Gamma\left(\frac{3}{2}+\frac{d}{4}\right)\frac{\|\nabla\phi\|_2^2}{\|\phi\|_2^2}t^{\frac{1}{2}-\frac{d}{4}}.
\end{split}\end{equation}
\end{cor}

\begin{proof}
Laplace transforming \eqref{Riesz-mean-ineq-DirichletbiLaplacian} yields inequality \eqref{Partition-function-inequality_bi} for all $t>0$. In view of the semiclassical expansion, we are interested in bounds for small $t$ and therefore we apply the inequality $1-x\leq  e^{-x}\leq 1$ for all $x\geq 0$ to \eqref{Partition-function-inequality_bi} and the inequality $\|\phi U_j\|_2^2\leq\|\phi\|_{\infty}^2$ and we get \eqref{part-fct-estimate-small-times_bi}.
\end{proof}

We prove now more explicit bounds presenting a first term which is sharp and a second term of the correct order in $k$ with respect to \eqref{weyl_dirichlet_biharmonic}. As we shall see, the more regular the domain $\Omega$ is, the more information is contained in the bounds. We note that formula \eqref{evsums-DirichletbiLaplacian1} with $\phi=1_{\Omega}$ is a ``reverse Berezin-Li-Yau inequality'' for the biharmonic operator.  Clearly, such an inequality does not hold and in fact we cannot use $\phi\equiv 1$ in \eqref{evsums-DirichletbiLaplacian1}. However, the form of inequality \eqref{evsums-DirichletbiLaplacian1} suggests that a suitable choice of $\phi$ is a function in $H^2_0(\Omega)\cap L^{\infty}(\Omega)$ which approximates the constant function $1$.

We construct now functions $\phi_h\in H^2_0(\Omega)\cap L^{\infty}(\Omega)$ depending on $h>0$ which approximate $1_{\Omega}$ as $h\rightarrow 0^+$ and with controlled $L^2(\Omega)$-norm of their gradients and Laplacians. For $x\in\mathbb R^d$ we denote by $\delta(x)$ the function $\delta(x):=\text{dist}(x,\partial\Omega)$. Let $h>0$ and let $\omega_h\subset\Omega$ be defined by
\begin{equation*}
\omega_h:=\left\{x\in\Omega:\delta(x)\leq h\right\}.
\end{equation*}
We note that $\omega_h=\Omega$ whenever $h\geq r_{\Omega}$, where $r_{\Omega}$ denotes the inradius of $\Omega$
\begin{equation*}
r_{\Omega}:=\underset{x\in\Omega}{\max}\underset{y\in\partial\Omega}{\min}|x-y|.
\end{equation*}

We define a function $\phi_h\in H^2_0(\Omega)$ such that $\phi_h\equiv 1$ in $\Omega\setminus\overline\omega_h$ and $0\le\phi_h\le1$ on $\Omega$ as follows. 
Let $f:[0,+\infty[\rightarrow\mathbb R$ be defined by
$$
f(r)=
\begin{cases}
\frac{d^2+6d+8}{8 B_d}(r^2-1)^2,& {\rm if\ }r\in[0,1[,\\
0,& {\rm if\ }r\in[1,+\infty[.
\end{cases}
$$
By construction $f\in C^{1,1}([0,+\infty))$, $f'(0)=0$, and $f(r)>0$ on $]0,1[$.

Let now $\eta_h:\mathbb R^d\to[0,+\infty[$ be defined by
$$
\eta_h(x):=\frac{1}{h^d}f\left(\frac{|x|}{h}\right).
$$
By construction
$$
\int_{\mathbb R^d}\eta_h(x)dx=1,
$$
for all $h>0$ and $\eta_h$ is supported on $B(0,h)$. 

Let $h<r_{\Omega}$ and consider $1_{h}:=1_{\Omega\setminus\overline{\omega_{h}}}$. We set
\begin{equation}\label{phi_h}
\phi_h:=\left(1_{\frac{h}{2}}\ast{\eta_{\frac{h}{2}}}\right){|_{\Omega}}.
\end{equation}
By construction, $\phi_h\in C^{1,1}(\mathbb R^d)$. Moreover, for any $x\in\mathbb R^d\setminus\Omega$,
$$
\phi_h(x)=\int_{\mathbb R^d}1_{\frac{h}{2}}(y)\eta_{\frac{h}{2}}(x-y)dy=\int_{B\left(x,\frac{h}{2}\right)}1_{\frac{h}{2}}(y)\eta_{\frac{h}{2}}(x-y)dy=0.
$$
In the same way one has that, for any $x\in\mathbb R^d\setminus\Omega$,
$$
\nabla\phi_h(x)=0.
$$
This means that $\phi_h$ is a continuosly differentible function on $\overline{\Omega}$ with ${\phi_h}{|_{\partial\Omega}}=|{\nabla\phi_h}{|_{\partial\Omega}}|=0$ and with Lipschitz continuous first partial derivatives, in other words, $\phi_h\in H^2_0(\Omega)$.

Moreover, for any $x\in\Omega\setminus\overline{\omega_h}$, we have
$$
\phi_h(x)=\int_{B\left(x,\frac{h}{2}\right)}1_{\frac{h}{2}}(y)\eta_{\frac{h}{2}}(x-y)dy=\int_{B\left(x,\frac{h}{2}\right)}\eta_{\frac{h}{2}}(x-y)dy=1,
$$
and, for any $x\in\omega_h$,
$$
0\leq\phi_h(x)=\int_{B\left(x,\frac{h}{2}\right)}1_{\frac{h}{2}}(y)\eta_{\frac{h}{2}}(x-y)dy\leq\int_{B\left(x,\frac{h}{2}\right)}\eta_{\frac{h}{2}}(x-y)dy=1.
$$
We estimate now the $L^{\infty}(\Omega)$-norm of $\nabla\phi_h$ and $\Delta\phi_h$ (note that, since $\phi_h\in C^{1,1}(\Omega)$, then $\Delta\phi_h\in L^{\infty}(\Omega)$). We have
\begin{multline*}
\left|\nabla\phi_h\right|^2=\sum_{i=1}^d\left|1_{\frac{h}{2}}\ast\partial_{x_i}\eta_{\frac{h}{2}}\right|^2\leq\|1_{\frac{h}{2}}\|_{\infty}^2\sum_{i=1}^d\left(\int_{B\left(0,\frac{h}{2}\right)}|\partial_{x_i}\eta_{\frac{h}{2}}|dx\right)^2\\
\leq\|\nabla\eta_{\frac{h}{2}}\|_2^2\left|B\left(0,\frac{h}{2}\right)\right|=\frac{8d(d+2)(d+4)}{d+6}\cdot\frac{1}{h^2},
\end{multline*}
hence
$$
\|\nabla\phi_h\|_{\infty}^2\leq \frac{8d(d+2)(d+4)}{d+6}\cdot\frac{1}{h^2}.
$$
Moreover
$$
\left|\Delta\phi_h\right|^2=\left|1_{\frac{h}{2}}\ast\Delta\eta_{\frac{h}{2}}\right|^2\leq\|1_{\frac{h}{2}}\|_{\infty}^2\left(\int_{B\left(0,\frac{h}{2}\right)}|\Delta\eta_{\frac{h}{2}}|dx\right)^2=64d^2(d+4)^2\left(\frac{d}{d+2}\right)^d\cdot\frac{1}{h^4},
$$
hence
$$
\|\Delta\phi_h\|_{\infty}^2\leq 64d^2(d+4)^2\left(\frac{d}{d+2}\right)^d\cdot\frac{1}{h^4}.
$$
We set
\begin{equation}
\label{adconst}
A_d^2:=\frac{8d(d+2)(d+4)}{d+6}\,,\ \ \ \tilde A_d^2:= 64d^2(d+4)^2\left(\frac{d}{d+2}\right)^d.
\end{equation}
We have proved the following
\begin{lem}\label{lemma_test}
Let $\Omega$ be a domain in $\mathbb R^d$ of finite measure. Let $r_{\Omega}>0$ denote the inradius of $\Omega$. Then, for all $h\in]0,r_{\Omega}]$ there exists a function $\phi_h\in H^2_0(\Omega)\cap L^{\infty}(\Omega)$ such that
\begin{enumerate}[i)]
\item $0\leq\phi_h(x)\leq 1$ for all $x\in\overline{\Omega}$;
\item $\phi_h\equiv 1$ on $\Omega\setminus\overline{\omega_h}$;
\item $\|\nabla\phi_h\|_{\infty}\leq A_d h^{-1}$, with $A_d$ depending only on $d$;
\item $\|\Delta\phi_h\|_{\infty}\leq \tilde A_d h^{-2}$, with $\tilde A_d$ depending only on $d$.
\end{enumerate}
\end{lem}

We note that, once we are able to estimate the size of $\omega_h$, by choosing a suitable $h$ into \eqref{Riesz-mean-ineq-DirichletbiLaplacian}, then inequalities \eqref{evsums-DirichletbiLaplacian1} and \eqref{Partition-function-inequality_bi} become asymptotically sharp. A suitable choice will be $h\sim k^{-1/d}$ in the case of sufficiently smooth domains. This is made clear in the next theorem,which is stated for averages of eigenvalues only. We remark though that analogous computations allow to prove related estimates for Riesz means and the partition functions as well.

\begin{thm}\label{main_upper_dirichlet}
Let $\Omega$ be a domain in $\mathbb R^d$ of finite measure.
\begin{enumerate}[i)]
\item For all positive integers $k$
\begin{equation}\label{rough_estimate_bilaplacian}
\frac{1}{k}\sum_{j=1}^k\Lambda_j\leq\frac{1}{r_{\Omega}^4}\left(\frac{d}{d+4}C_d^2\left(a_d|\Omega|\right)^{\frac{4}{d}}\left(\frac{k}{|\Omega|}\right)^{\frac{4}{d}}+2C_d\left(b_d|\Omega|\right)^{\frac{2}{d}}\left(\frac{k}{|\Omega|}\right)^{\frac{2}{d}}+c_d\right),
\end{equation}
where $a_d,b_d,c_d$ are constants which depend only on the dimension and are given by \eqref{cd}.
\item Let $\Omega$ be such that $\lim_{h\rightarrow 0^+}\frac{|\omega_h|}{h}=|\partial\Omega|<\infty$. Then, for any $k\geq |\Omega|\left(\frac{\sqrt{d}A_d}{2C_d^{1/2}r_{\Omega}}\right)^{d}$,
\begin{equation}\label{explicit_sum}
 \frac{1}{k}\sum_{j=1}^k\Lambda_j\leq \frac{d}{d+4}C_d^2\left(\frac{k}{|\Omega|}\right)^{\frac{4}{d}}
+M_d\frac{|\partial\Omega|}{|\Omega|}C_d^{\frac{3}{2}}\left(\frac{k}{|\Omega|}\right)^{\frac{3}{d}}+R(k),
\end{equation}
where $M_d$ depends only on $d$ and is given by \eqref{Md}. Here $R(k)=o(k^{3/d})$ as $k\rightarrow +\infty$, and it depends explicitly on $k,d,|\Omega|,|\partial\Omega|$ and $|\omega_{h(k)}|$ with $h(k)$ given by \eqref{hk}.
\item If $\Omega$ is convex or if $\Omega$ is of class $C^2$ and bounded, then $ii)$ holds. Moreover there exists $C=C(d,\Omega)>0$ such that \eqref{explicit_sum} holds with
\begin{equation}\label{rem_est}
R(k)\leq C \left(\frac{k}{|\Omega|}\right)^{\frac{2}{d}}.
\end{equation}
Finally, there exists $k_0$ which depends only on $d$ and $\Omega$ such that, for all $k\geq k_0$, $R(k)$ depends explicitly on $k,d,|\Omega|,|\partial\Omega|$ if $\Omega$ is convex and on $k,d,|\Omega|,|\partial\Omega|$ and integrals of the mean curvature $\mathcal H$ of $\partial\Omega$ if $\Omega$ is of class $C^2$.
\end{enumerate}
\end{thm}

\begin{proof}
We start by proving $i)$. We construct a test function in $H^2_0(\Omega)$ supported in a ball $B_{r_{\Omega}}$ of radius $r_{\Omega}$ contained in $\overline\Omega$ (by definition of $r_{\Omega}$ such a ball exists). Let then
$$
\psi_{r_{\Omega}}(x):=\left(\frac{|x|^2}{r_{\Omega}^2}-1\right)^2.
$$
Explicit computations show that
\begin{equation*}\begin{split}
\|\psi_{r_{\Omega}}\|_2^2= &\frac{384r_{\Omega}^d B_d}{(d+2)(d+4)(d+6)(d+8)},\\
\|\psi_{r_{\Omega}}\|_{\infty}^2=&1,\\
\frac{\|\nabla\psi_{r_{\Omega}}\|_2^2}{\|\psi_{r_{\Omega}}\|_2^2}=&\frac{d(d+8)}{3r_{\Omega}^2},\\
\frac{\|\Delta\psi_{r_{\Omega}}\|_2^2}{\|\psi_{r_{\Omega}}\|_2^2}=&\frac{(8+d(d-2))(d+6)(d+8)}{6r_{\Omega}^4}.
\end{split}\end{equation*}
We set 
\begin{equation}\label{cd}
\begin{split}
a_d =&\frac{(d+2)(d+4)(d+6)(d+8)}{384 B_d},\\
b_d =& a_d\left(\frac{d(d+8)}{3}\right)^{\frac{d}{2}},\\
c_d =&\frac{(8+d(d-2))(d+6)(d+8)}{6}.
\end{split}
\end{equation}
Formula \eqref{rough_estimate_bilaplacian} now follows from \eqref{evsums-DirichletbiLaplacian1} with $\phi=\psi_{r_{\Omega}}$ and standard computations. This proves point $i)$.

We prove now $ii)$. From \eqref{evsums-DirichletbiLaplacian1} it follows that, for $d\geq 4$,
\begin{equation}\label{evsums-DirichletbiLaplacian1_step1}
\begin{split}
 \frac{1}{k}\sum_{j=1}^k\Lambda_j & \leq \frac{d}{d+4}C_d^2\left(\frac{k}{|\Omega|}\right)^{\frac{4}{d}}\rho(\phi)^{-\frac{4}{d}}\\
& \quad +2\frac{\|\nabla\phi\|_2^2}{\|\phi\|_2^2}C_d\left(\frac{k}{|\Omega|}\right)^{\frac{2}{d}}\rho(\phi)^{-\frac{2}{d}}
+\frac{\|\Delta\phi\|_2^2}{\|\phi\|_2^2}\\
& = \frac{d}{d+4}C_d^2\left(\frac{k}{|\Omega|}\right)^{\frac{4}{d}}+\frac{d}{d+4}C_d^2\left(\frac{k}{|\Omega|}\right)^{\frac{4}{d}}\left(\rho(\phi)^{-4/d}-1\right)\\
& \quad +2\frac{\|\nabla\phi\|_2^2}{\|\phi\|_2^2}C_d\left(\frac{k}{|\Omega|}\right)^{\frac{2}{d}}\rho(\phi)^{-\frac{2}{d}}
+\frac{\|\Delta\phi\|_2^2}{\|\phi\|_2^2}\\
& \leq \frac{d}{d+4}C_d^2\left(\frac{k}{|\Omega|}\right)^{\frac{4}{d}}+\frac{4}{d+4}C_d^2\left(\frac{k}{|\Omega|}\right)^{\frac{4}{d}}\left(\rho(\phi)^{-1}-1\right)\\
& \quad +2\frac{\|\nabla\phi\|_2^2}{\|\phi\|_2^2}C_d\left(\frac{k}{|\Omega|}\right)^{\frac{2}{d}}\rho(\phi)^{-\frac{2}{d}}
+\frac{\|\Delta\phi\|_2^2}{\|\phi\|_2^2}\\
& =  \frac{d}{d+4}C_d^2\left(\frac{k}{|\Omega|}\right)^{\frac{4}{d}}+\frac{4}{d+4}C_d^2\left(\frac{k}{|\Omega|}\right)^{\frac{4}{d}}\left(\frac{|\Omega|\|\phi\|_{\infty}^2-\|\phi\|_2^2}{\|\phi\|_2^2}\right)\\
& \quad +2\frac{\|\nabla\phi\|_2^2}{\|\phi\|_2^2}C_d\left(\frac{k}{|\Omega|}\right)^{\frac{2}{d}}\left(\frac{|\Omega|\|\phi\|_{\infty}^2}{\|\phi\|_2^2}\right)^{\frac{2}{d}}
+\frac{\|\Delta\phi\|_2^2}{\|\phi\|_2^2},
\end{split}
\end{equation}
where $\rho(\phi)$ is defined in \eqref{rhophi}. We have used Bernoulli's inequality in the fifth line of \eqref{evsums-DirichletbiLaplacian1_step1}. If $d=2,3$, we use the following fact
$$
\rho(\phi)^{-4/d}-1=(\rho(\phi)^2-1+1)^{-2/d}-1\leq\frac{2}{d}(\rho(\phi)^{-2}-1)=\frac{2}{d}(\rho(\phi)^{-1}+1)(\rho(\phi)^{-1}-1),
$$
so that, for $d=2,3$ we have
\begin{multline}\label{evsums-DirichletbiLaplacian1_step11}
 \frac{1}{k}\sum_{j=1}^k\Lambda_j\leq \frac{d}{d+4}C_d^2\left(\frac{k}{|\Omega|}\right)^{\frac{4}{d}}
+\frac{2}{d+4}C_d^2\left(\frac{k}{|\Omega|}\right)^{\frac{4}{d}}\left(\frac{|\Omega|\|\phi\|_{\infty}^2+\|\phi\|_2^2}{\|\phi\|_2^2}\right)\left(\frac{|\Omega|\|\phi\|_{\infty}^2-\|\phi\|_2^2}{\|\phi\|_2^2}\right)\\
+2\frac{\|\nabla\phi\|_2^2}{\|\phi\|_2^2}C_d\left(\frac{k}{|\Omega|}\right)^{\frac{2}{d}}\left(\frac{|\Omega|\|\phi\|_{\infty}^2}{\|\phi\|_2^2}\right)^{\frac{2}{d}}
+\frac{\|\Delta\phi\|_2^2}{\|\phi\|_2^2}.
\end{multline}
For each positive integer $k$, we choose $\phi=\phi_h$ defined by \eqref{phi_h} into \eqref{evsums-DirichletbiLaplacian1_step1} and \eqref{evsums-DirichletbiLaplacian1_step11}. Thanks to Lemma \ref{lemma_test} and to the fact that $|\Omega|-|\omega_h|\leq\|\phi_h\|_2^2\leq |\Omega|$, we have
\begin{multline}\label{evsums-DirichletbiLaplacian1_step21}
 \frac{1}{k}\sum_{j=1}^k\Lambda_j\leq \frac{d}{d+4}C_d^2\left(\frac{k}{|\Omega|}\right)^{\frac{4}{d}}
+\frac{4}{d+4}C_d^2\left(\frac{k}{|\Omega|}\right)^{\frac{4}{d}}\left(\frac{|\omega_h|}{|\Omega|-|\omega_h|}\right)\\
+2\frac{A_d^2|\omega_h|}{h^2(|\Omega|-|\omega_h|)}C_d\left(\frac{k}{|\Omega|}\right)^{\frac{2}{d}}\left(\frac{|\Omega|}{|\Omega|-|\omega_h|}\right)^{\frac{2}{d}}
+\frac{\tilde A_d^2|\omega_h|}{h^4(|\Omega|-|\omega_h|)}.
\end{multline}
if $d\geq 4$, and
\begin{multline}\label{evsums-DirichletbiLaplacian1_step22}
 \frac{1}{k}\sum_{j=1}^k\Lambda_j\leq \frac{d}{d+4}C_d^2\left(\frac{k}{|\Omega|}\right)^{\frac{4}{d}}
+\frac{2}{d+4}C_d^2\left(\frac{k}{|\Omega|}\right)^{\frac{4}{d}}\left(\frac{2|\Omega|}{|\Omega|-|\omega_h|}\right)\left(\frac{|\omega_h|}{|\Omega|-|\omega_h|}\right)\\
+2\frac{A_d^2|\omega_h|}{h^2(|\Omega|-|\omega_h|)}C_d\left(\frac{k}{|\Omega|}\right)^{\frac{2}{d}}\left(\frac{|\Omega|}{|\Omega|-|\omega_h|}\right)^{\frac{2}{d}}
+\frac{\tilde A_d^2|\omega_h|}{h^4(|\Omega|-|\omega_h|)},
\end{multline}
if $d=2,3$.
In both cases, since $\lim_{h\rightarrow 0^+}\frac{|\omega_h|}{h}=|\partial\Omega|$, we can write
\begin{multline}\label{evsums-DirichletbiLaplacian1_step3}
 \frac{1}{k}\sum_{j=1}^k\Lambda_j\leq \frac{d}{d+4}C_d^2\left(\frac{k}{|\Omega|}\right)^{\frac{4}{d}}
+\frac{4}{d+4}C_d^2\left(\frac{k}{|\Omega|}\right)^{\frac{4}{d}}\left(\frac{h|\partial\Omega|}{|\Omega|}\right)\\
+2\frac{A_d^2|\partial\Omega|}{h|\Omega|}C_d\left(\frac{k}{|\Omega|}\right)^{\frac{2}{d}}
+\frac{\tilde A_d^2|\partial\Omega|}{h^3|\Omega|}+R(k,h),
\end{multline}
where $R(k,h)$ is defined in \eqref{remainder} and the constants $A_d, \tilde A_d$ are as in \eqref{adconst}. Here we could optimize with respect ot $h$ and find the optimal $h$ which is given by an explicit dimensional constant times $C_d^{-\frac{1}{2}}\left(\frac{k}{|\Omega|}\right)^{-\frac{1}{d}}$. We set
\begin{equation}\label{hk}
h=h(k)=\sqrt{\frac{d+4}4}A_d C_d^{-\frac{1}{2}}\left(\frac{k}{|\Omega|}\right)^{-\frac{1}{d}}\varepsilon,
\end{equation}
so that inequality \eqref{evsums-DirichletbiLaplacian1_step3} becomes
\begin{multline*}
 \frac{1}{k}\sum_{j=1}^k\Lambda_j\leq \frac{d}{d+4}C_d^2\left(\frac{k}{|\Omega|}\right)^{\frac{4}{d}}\\
+\sqrt{\frac{4}{d+4}}A_dC_d^{\frac 3 2}\left(\frac{k}{|\Omega|}\right)^{\frac{3}{d}}\frac{|\partial\Omega|}{|\Omega|}\left(\varepsilon+\frac 2 \varepsilon +\frac 4 {d+4}\frac {\tilde A_d^2}{A_d^4}\varepsilon^{-3}\right)
+R(k,h(k)).
\end{multline*}
For simplicity we choose $\varepsilon=\sqrt{2}$ which optimizes the first two terms depending on $\varepsilon$ since the goal is not to get best constants here (already the constants $A_d, \tilde A_d$ are not optimal). It follows that, for any $k\geq |\Omega|\frac{A_d^d}{r_{\Omega}^d}\left(\frac{d+4}{2C_d}\right)^{d /2}$ (we need $h\leq r_{\Omega}$), we obtain
\begin{equation*}
 \frac{1}{k}\sum_{j=1}^k\Lambda_j\leq \frac{d}{d+4}C_d^2\left(\frac{k}{|\Omega|}\right)^{\frac{4}{d}}
+M_d\frac{|\partial\Omega|}{|\Omega|}C_d^{\frac{3}{2}}\left(\frac{k}{|\Omega|}\right)^{\frac{3}{d}}+R(k,h(k)),
\end{equation*}
where
\begin{equation}\label{Md}
M_d=8\left(\frac{d(d+2)}{d+6}\right)^{\frac 1 2}\left(2+\frac{(d+6)^2}{(d+2)^2(d+4)}\left(\frac d {d+2}\right)^d\right).
\end{equation}
We also note that the remainder function $R(k,h(k))$ in \eqref{evsums-DirichletbiLaplacian1_step3} with $h=h(k)$ given by \eqref{hk} is $o(k^{3/d})$ as $k\rightarrow +\infty$. This concludes the proof of $ii)$.

Now, let us pass to $iii)$. It is known that $\lim_{h\rightarrow 0^+}\frac{|\omega_h|}{h}=|\partial\Omega|$ if $\Omega$ is Lipschitz (see e.g., \cite{colesanti_ambrosio}). In particular this is true for convex sets (which are Lipschitz) and for sets with $C^2$ boundaries. Hence $ii)$ holds for these classes of domains. Let us now write explicitly the remainder $R(k,h)$ in \eqref{evsums-DirichletbiLaplacian1_step3}. For simplicity we consider the case $d\geq 4$, the case $d=2,3$ being similar. We have that
\begin{multline}\label{remainder}
R(k,h)=\frac{d}{d+4}C_d^2\left(\frac{k}{|\Omega|}\right)^\frac{4}{d}\left(\frac{|\omega_h|}{|\Omega|-|\omega_h|}-\frac{h|\partial\Omega|}{|\Omega|}\right)\\
+2\frac{A_d^2}{h^2}C_d\left(\frac{k}{|\Omega|}\right)^{\frac{2}{d}}\left(\frac{|\omega_h|}{\left(1-\frac{|\omega_h|}{|\Omega|}\right)^{2/d}(|\Omega|-|\omega_h|)}-\frac{h|\partial\Omega|}{|\Omega|}\right)
+\frac{\tilde A_d^2}{h^4}\left(\frac{|\omega_h|}{|\Omega|-|\omega_h|}-\frac{h|\partial\Omega|}{|\Omega|}\right).
\end{multline}
Consider $\Omega$ convex first. We note that for convex domains and for all $h\leq r_{\Omega}$, then $|\omega_h|\leq h|\partial\Omega|$. This follows from the co-area formula and from the fact that the measure of the sets $\partial\Omega_h:=\left\{x\in\Omega:{\rm dist}(x,\partial\Omega)=h\right\}$ is a non-increasing function of $h$, for $h\in[0, r_{\Omega}]$. Hence, from \eqref{remainder} we deduce that
\begin{equation*}
\begin{split}
R(k,h) & \leq\frac{d}{d+4}C_d^2\left(\frac{k}{|\Omega|}\right)^\frac{4}{d}\frac{h^2|\partial\Omega|^2}{|\Omega|(|\Omega|-h|\partial\Omega|)}\\
& \quad +2\frac{A_d^2}{h^2}C_d\left(\frac{k}{|\Omega|}\right)^{\frac{2}{d}}h|\partial\Omega|\left(\frac{1}{|\Omega|\left(1-\frac{h|\partial\Omega|}{|\Omega|}\right)^{1+2/d}}-\frac{1}{|\Omega|}\right)
+\frac{\tilde A_d^2}{h^4}\frac{h^2|\partial\Omega|^2}{|\Omega|(|\Omega|-h|\partial\Omega|)}\\
& \leq \frac{d}{d+4}C_d^2\left(\frac{k}{|\Omega|}\right)^\frac{4}{d}\frac{h^2|\partial\Omega|^2}{|\Omega|(|\Omega|-h|\partial\Omega|)}
+2\frac{A_d^2}{h^2}C_d\left(\frac{k}{|\Omega|}\right)^{\frac{2}{d}}\frac{\left(1+\frac{2}{d}\right)h^2|\partial\Omega|^2}{|\Omega|\left(|\Omega|-h|\partial\Omega|\left(1+\frac{2}{d}\right)\right)}\\
& \quad +\frac{\tilde A_d^2}{h^4}\frac{h^2|\partial\Omega|^2}{|\Omega|(|\Omega|-h|\partial\Omega|)},
\end{split}
\end{equation*}
where the second inequality follows from Bernoulli's inequality. Choosing $h=h(k)$ as in point $ii)$ (see \eqref{hk}) we immediately deduce the validity of $iii)$ in the case that $\Omega$ is convex. 

Let now $\Omega$ be of class $C^2$ and bounded. In this case, we note that there exists $\bar h\in]0, r_{\Omega}[$ such that any point in $\omega_h$ has a unique nearest point on $\partial\Omega$, for all $h\in(0,\bar h)$. Let us take the supremum of such $\bar h$ (still denoted $\bar h$). It is standard to see that for all $h\in]0,\bar h[$
\begin{equation}\label{tubular_smooth}
|\omega_h|\leq h|\partial\Omega|+\frac{h^2}{d}\sum_{j=2}^d\binom{d}{j}(-1)^{j-1}h^{j-2}\int_{\partial\Omega}\mathcal H(s)^{j-1}d\sigma(s),
\end{equation}
where $\mathcal H(s)$ denotes the mean curvature of $\partial\Omega$ at $s\in\partial\Omega$. We refer to \cite[Theorem 2.19]{HaPrSt18} for a proof of \eqref{tubular_smooth}. We choose again $h=h(k)$ as in \eqref{hk} and insert it into \eqref{evsums-DirichletbiLaplacian1_step3}. Therefore, we are allowed to implement the upper bound \eqref{tubular_smooth} into \eqref{remainder}. This confirms the claim of $iii)$ for bounded domains of class $C^2$.
\end{proof}

We conclude this discussion with a few remarks.

\begin{rem}
Point $i)$ of Theorem \ref{main_upper_dirichlet} provides a bound which is not asymptotically sharp in $k$ and which shows a dependence on $ r_{\Omega}$. The presence of the term $ r_{\Omega}^{-1/4}$ is somehow natural for lower eigenvalues. For example, for $d=2$ it is known that
$$
\lambda_1\geq\frac{1}{2\gamma r_{\Omega}^2},
$$ 
if $\gamma\ge 2$, where $\gamma$ denotes the number of connected components of $\partial\Omega$ (see \cite{croke}), while
$$
\lambda_1\geq\frac{1}{4 r_{\Omega}^2},
$$
if $\gamma=1$. Since $\Lambda_1\geq\lambda_1^2$, the exponent $4$ on $ r_{\Omega}$ in \eqref{rough_estimate_bilaplacian} is sharp. However, for larger eigenvalues the bound \eqref{rough_estimate_bilaplacian} is not good, and in fact asymptotically sharp bounds hold starting from a given positive integer $k_0$ depending on $d$ and $\Omega$, as in \eqref{explicit_sum} (cf.\ \eqref{rough_o}).
\end{rem}

\begin{rem}
Point $ii)$ of Theorem \ref{main_upper_dirichlet} holds if $\Omega$ is such that $\mathcal M_{\Omega}(\partial\Omega)=|\partial\Omega|$, where
\begin{equation}\label{tube_vol}
\mathcal M_{\Omega}(\partial\Omega)=\lim_{h\rightarrow 0^+}\frac{|\omega_h|}{h}.
\end{equation}
The limit \eqref{tube_vol} is usually called the Minkowski content of $\partial\Omega$ {relative to} $\Omega$ (see e.g., \cite{lapidus_fractal_1,lapidus_fractal_2}). There are some sufficient conditions which assure that $\mathcal M_{\Omega}(\partial\Omega)=|\partial\Omega|$, for example if $\Omega$ has a Lipschitz boundary (see \cite{colesanti_ambrosio} for the proof and for a more detailed discussion on Minkowski content and conditions ensuring $\mathcal M_{\Omega}(\partial\Omega)=|\partial\Omega|$).
\end{rem}

\begin{rem}
The estimate \eqref{rem_est} of point $iii)$ can be proved also for Lipschitz domains with piecewise $C^2$ boundaries. In addition, more refined estimates for the remainder in the case of smooth, mean convex or convex sets can be obtained by means of a deeper (though long and technical) analysis (see e.g., \cite{HaPrSt18}). In dimension $d=2$ we can find explicit dependence of the remainder $R(k)$ in terms of the number of connected components of the boundary (for $C^2$ domains) or in terms of the angles (in the case of polygons), see \cite{HaPrSt18}. We don't enter here into the details of more refined estimates, which require more careful but standard computations. However, we remark that Theorem \ref{dirichlet_bilaplacian_thm_general} gives a general recipe to obtain asymptotically sharp upper bounds for averages with explicit dependence on the geometry of $\Omega$ (via a suitable choice of test functions $\phi$).

We also remark that asymptotically sharp estimates with a well-behaved second term can be obtained for Riesz means and for the partition function by plugging into \eqref{Riesz-mean-ineq-DirichletbiLaplacian} and \eqref{evsums-DirichletbiLaplacian1} the same test functions $\phi_h$ used in the proof of Theorem \ref{main_upper_dirichlet}.
\end{rem}

\begin{rem}
We note that the second term in the upper bound \eqref{explicit_sum} coincides with the second term of the semiclassical asymptotic expansion of the average of biharmonic Dirichlet eigenvalues \eqref{weyl_dirichlet_biharmonic}, up to a multiplicative dimensional constant.
\end{rem}

\begin{rem}
We observe that formula \eqref{evsums-DirichletbiLaplacian1_step22} holds for any $\Omega\subset\mathbb R^d$ of finite measure (it need not be bounded), hence upper bounds depend on information of $|\omega_h|$. In a general situation we can only say that $|\omega_h|\rightarrow 0$ as $h\rightarrow 0^+$. This is a simple consequence of the Dominated Convergence Theorem. We deduce that $|\omega_h|=\omega(h)$ where $\omega:]0,+\infty[\rightarrow\mathbb R$ is such that $\lim_{h\rightarrow 0^+}\omega(h)=0$. As in the proof of point $ii)$ of Theorem \ref{main_upper_dirichlet}, we can prove that, for any $\Omega$ of finite measure (we take for simplicity $h=h(k)=C_d^{-1/2}\left(\frac{k}{|\Omega|}\right)^{-1/d}$ into \eqref{evsums-DirichletbiLaplacian1_step21}-\eqref{evsums-DirichletbiLaplacian1_step22})
\begin{multline}\label{rough_o}
\frac{1}{k}\sum_{j=1}^k\Lambda_j\leq\frac{d}{d+4}C_d^2\left(\frac{k}{|\Omega|}\right)^{\frac{4}{d}}+\frac{M_d'}{\Omega}C_d^2\left(\frac{k}{|\Omega|}\right)^{\frac{4}{d}}
\omega\left(C_d^{-1/2}\left(\frac{k}{|\Omega|}\right)^{-1/d}\right)\\
+o\left(\left(\frac{k}{|\Omega|}\right)^{\frac{4}{d}}\omega\left(C_d^{-1/2}\left(\frac{k}{|\Omega|}\right)^{-1/d}\right)\right),
\end{multline}
as $k\rightarrow+\infty$, for all $k\geq |\Omega|C_d^{-d/2} r_{\Omega}^{-d}$. Here $M_d'$ is a constant which depends only on the dimension and which can be computed explicitly as in the proof of point $ii)$ of Theorem \ref{main_upper_dirichlet}. Combining \eqref{rough_o} with the Berezin-Li-Yau inequality
$$
\frac{1}{k}\sum_{j=1}^k\Lambda_j\geq\frac{d}{d+4}C_d^2\left(\frac{k}{|\Omega|}\right)^{\frac{4}{d}}
$$
proved in \cite{Lap1997} for all domains of finite measure, we deduce the validity of \eqref{weyllaweig} for $\omega_j=\Lambda_j$ on domains of finite measure.
\end{rem}

\begin{rem}

Now, let us denote by $D$ the Minkowski dimension of $\partial\Omega$ relative to $\Omega$, which is defined by
$$
D:=\inf\left\{\beta\in[d-1,d]:\lim_{h\rightarrow 0^+}\frac{|\omega_h|}{h^{d-\beta}}<+\infty\right\}.
$$
Let the $D$-dimensional Minkowski content of $\partial\Omega$ relative to $\Omega$ be defined by
$$
\mathcal M_D(\partial\Omega):=\lim_{h\rightarrow 0^+}\frac{|\omega_h|}{h^{d-D}}.
$$
Assume now that $\Omega$ is such that the Minkowski dimension of $\partial\Omega$ relative to $\Omega$ is $D\in]d-1,d[$ (for example, if $\Omega$ is a fractal set) and let $\mathcal M_D(\partial\Omega)$ be the Minkowski content of $\partial\Omega$ relative to $\Omega$. From \eqref{rough_o} we immediately see that
\begin{equation*}
\frac{1}{k}\sum_{j=1}^k\Lambda_j
\leq\frac{d}{d+4}C_d^2\left(\frac{k}{|\Omega|}\right)^{\frac{4}{d}}
+M_d'\frac{\mathcal M_D(\partial\Omega)}{|\Omega|}C_d^{\frac{D-d+4}{2}}\left(\frac{k}{|\Omega|}\right)^{\frac{D-d+4}{d}}
+o\left(\left(\frac{k}{|\Omega|}\right)^{\frac{D-d+4}{d}}\right),
\end{equation*}
as $k\rightarrow+\infty$, for all $k\geq |\Omega|C_d^{-d/2} r_{\Omega}^{-d}$. Hence the second term of the upper bound for the average depends only on $k,d,D,|\Omega|$ and $\mathcal M_D(\partial\Omega)$. Analogous inequalities have been proved for the eigenvalues of the Dirichlet Laplacian (see \cite{HaPrSt18}), and are related to the so-called Weyl-Berry conjecture (see \cite{carmona_fractal,lapidus_fractal_2}). 
\end{rem}


\subsection{Asymptotically Weyl-sharp bounds on eigenvalues}
Assume that $\Omega$ is such that
\begin{equation}\label{dirichlet_ineq_1}
\frac{1}{k}\sum_{j=1}^k\Lambda_j\geq\frac{d}{d+4}C_d^2\left(\frac{k}{|\Omega|}\right)^{\frac{4}{d}}
\end{equation}
and
\begin{equation}\label{dirichlet_ineq_2}
\frac{1}{k}\sum_{j=1}^k\Lambda_j\leq\frac{d}{d+4}C_d^2\left(\frac{k}{|\Omega|}\right)^{\frac{4}{d}}+A \left(\frac{k}{|\Omega|}\right)^{\frac{3}{d}}
\end{equation}
for some constant $A$ independent of $k$, for all $k\geq k_0$ (this is for example the case of point $ii)$ of Theorem \ref{main_upper_dirichlet}). Then
\begin{multline}\label{dirichlet_ineq_1_2}
\Lambda_k\geq C_d^2\left(\frac{k}{|\Omega|}\right)^{\frac{4}{d}}
-\left(\frac{6(d+1)}{d(d+4)}\frac{C_d^2}{|\Omega|^{\frac{4}{d}}}+2\frac{A}{|\Omega|^{\frac{3}{d}}}\right)k^{\frac{7}{2d}}\\
+\left(\frac{C_d^2}{d(d+4)|\Omega|^{\frac{4}{d}}}+\frac{d+3}{d}\frac{A}{|\Omega|^{\frac{3}{d}}}\right)k^{\frac{3}{d}}
-\frac{3}{2}\left(\frac{9+12d}{4d^2}\right)\frac{k^{\frac{5}{2d}}}{|\Omega|^{\frac{3}{d}}}
+\frac{9A}{16d^2}\frac{k^{\frac{2}{d}}}{|\Omega|^{\frac{3}{d}}},
\end{multline}
and
\begin{multline}\label{dirichlet_ineq_2_2}
\Lambda_{k+1}\leq C_d^2\left(\frac{k}{|\Omega|}\right)^{\frac{4}{d}}
+\left(\frac{6(d+1)}{d(d+4)}\frac{C_d^2}{|\Omega|^{\frac{4}{d}}}+2\frac{A}{|\Omega|^{\frac{3}{d}}}\right)k^{\frac{7}{2d}}\\
+\left(\frac{9C_d^2}{d(d+4)|\Omega|^{\frac{4}{d}}}+\frac{d+3}{d}\frac{A}{|\Omega|^{\frac{3}{d}}}\right)k^{\frac{3}{d}}
+\frac{3}{2}\left(\frac{9+12d}{4d^2}\right)\frac{k^{\frac{5}{2d}}}{|\Omega|^{\frac{3}{d}}}
+\frac{81A}{16d^2}\frac{k^{\frac{2}{d}}}{|\Omega|^{\frac{3}{d}}}.
\end{multline}

In particular, for all $k\geq k_0$ there exists a constant $C(d,|\Omega|,A)$ such that
\begin{equation}\label{modulus_dirichlet}
\left|\Lambda_k-C_d^2\left(\frac{k}{|\Omega|}\right)^{\frac{4}{d}}\right|\leq C(d,|\Omega|,A)k^{\frac{7}{2d}}.
\end{equation}
Inequalities \eqref{dirichlet_ineq_1_2} and \eqref{dirichlet_ineq_2_2} follow from \eqref{dirichlet_ineq_1} and \eqref{dirichlet_ineq_2} by observing that
$$
\Lambda_k\geq\frac{1}{l}\sum_{j=k-l+1}^k\Lambda_j
$$
and
$$
\Lambda_{k+1}\leq\frac{1}{l}\sum_{j=k+1}^{k+l}\Lambda_j,
$$
and by choosing $l\in\mathbb N$ such that
$$
l=k^{1-\frac{1}{2d}}+b
$$
with $b\in\left[-\frac{1}{2},\frac{1}{2}\right]$. In particular, with this choice, 
\begin{equation}\label{singe_step_4}
\dfrac{1}{2}k^{1-\dfrac{1}{2d}}\leq l\leq \dfrac{3}{2}k^{1-\dfrac{1}{2d}},
\end{equation}
and $k-1\leq l\leq k+1$. For example, we see that
\begin{equation}\label{singe_step_1}
\begin{split}
\Lambda_k & \geq\frac{1}{l}\sum_{j=k-l+1}^k\Lambda_j=\frac{1}{l}\left(\sum_{j=1}^k\Lambda_j-\sum_{j=1}^{k-l}\Lambda_j\right)\\
& \geq\left(\frac{d}{d+4}\frac{C_d^2}{|\Omega|^{\frac{4}{d}}}\frac{1}{l}\left(k^{1+\frac{4}{d}}-(k-l)^{1-\frac{4}{d}}\right)-\frac{A}{|\Omega|^{\frac{3}{d}}}\frac{(k-l)^{1+\frac{3}{d}}}{l}\right)\\
& =\frac{d}{d+4}C_d^2\left(\frac{k}{|\Omega|}\right)^{\frac{4}{d}}\frac{k}{l}\left(1-\left(1-\frac{l}{k}\right)^{1+\frac{4}{d}}\right)
-A\left(\frac{k}{|\Omega|}\right)^{\frac{3}{d}}\frac{k}{l}\left(1-\left(1-\frac{l}{k}\right)^{1+\frac{3}{d}}\right).
\end{split}
\end{equation}
We also have
\begin{multline}\label{singe_step_2}
\frac{k}{l}\left(1-\left(1-\frac{l}{k}\right)^{1+\frac{4}{d}}\right)=\frac{k}{l}\left(1-\left(1-\frac{l}{k}\right)\left(\left(1-\frac{l}{k}\right)^{\frac{2}{d}}\right)^2\right)\\
\geq \frac{k}{l}\left(1-\left(1-\frac{l}{k}\right)\left(1-\frac{2l}{dk}\right)^2\right)=\frac{d+4}{d}-\frac{4(d+1)l}{d^2k}+\frac{4l^2}{d^2k^2},
\end{multline}
and similarly
\begin{equation}\label{singe_step_3}
\frac{k}{l}\left(1-\left(1-\frac{l}{k}\right)^{1+\frac{3}{d}}\right)\leq\frac{k}{l}-\frac{3+d}{d}+\frac{(9+12d)l}{4d^2k}-\frac{9l^2}{4d^2k^2}.
\end{equation}
Bound \eqref{dirichlet_ineq_1_2} follows by plugging \eqref{singe_step_2} and \eqref{singe_step_3} into \eqref{singe_step_1} and by \eqref{singe_step_4}. The upper bound \eqref{dirichlet_ineq_2_2} is proven similarly.


\section{The biharmonic Navier and Kuttler-Sigillito operators}

In this section we focus our attention to the Navier \eqref{equazione},\eqref{IBC}, and Kuttler-Sigillito \eqref{equazione},\eqref{KBC} problems. In particular, the quadratic form \eqref{qf} will be set into $H^2(\Omega)\cap H^1_0(\Omega)$ for the Navier problem, and into $H^2_{\nu}(\Omega)$ for the Kuttler-Sigillito problem, for $\Omega\subset \mathbb R^d$ a bounded open set. We observe that for the eigenvalues of the biharmonic operator with Navier and Kuttler-Sigillito boundary conditions, the very same lower bounds on Riesz means (upper bounds on averages) of the Dirichlet case hold. In particular we have the following result, which is valid for any domain $\Omega$ with finite Lebesgue measure (for the Navier problem) and with Lipschitz boundary (for the Kuttler-Sigillito problem).

\begin{thm}
Let $a\in(-(d-1)^{-1},1)$.
\begin{enumerate}[i)]
\item  Theorem \ref{dirichlet_bilaplacian_thm_general}, Corollary \ref{dirichlet_bounds_spec_fcn_bi} and Theorem \ref{main_upper_dirichlet} hold with $\Lambda_j,U_j$ replaced by $\tilde\Lambda_j(a),\tilde U_j$ or $\tilde M_j(a),\tilde V_j$.
\item Formulas \eqref{dirichlet_ineq_1_2}, \eqref{dirichlet_ineq_2_2} and \eqref{modulus_dirichlet} hold with $\Lambda_j$ replaced by $\tilde\Lambda_j(a)$ or $\tilde M_j(a)$.
\end{enumerate}
\end{thm}
\begin{proof}
As for point $i)$, the proofs of Theorem \ref{dirichlet_bilaplacian_thm_general}, Corollary \ref{dirichlet_bounds_spec_fcn_bi} and Theorem \ref{main_upper_dirichlet} in the case of Navier or Kuttler-Sigillito conditions can be carried out exactly in the same way as in the Dirichlet case by using test functions $\phi\in H^2_0(\Omega)$. Also, those arguments yield the same results for the Navier case when using $\phi\in H^2(\Omega)\cap H^1_0(\Omega)$. Alternatively, the results immediately follow by pointwise comparison of eigenvalues:
$$
\tilde M_j(a),\tilde\Lambda_j(a)\leq\Lambda_j,
$$
for all positive integers $j$, see \eqref{fullchain}. Point $ii)$ follows from point $i)$ as in the Dirichlet case.
\end{proof}

We also have the following inequalities relating Navier eigenvalues to Laplacian eigenvalues. Note that, if $\Omega$ satisfies the uniform outer ball condition, the inequalities are valid also for $a=1$.

\begin{thm}
Let $a\in(-(d-1)^{-1},1)$. For all positive integers $m,n,N$
\begin{equation}\label{prima}
  \sum_{j=1}^{n}(\lambda_{n+1}-\lambda_j)\geq \sum_{k=1}^{N}(\lambda_{n+1}-\int_{\Omega}|\nabla \tilde{U}_k|^2\,dx),
\end{equation}
\begin{equation}\label{seconda}
  \sum_{j=2}^{m}(\mu_{m+1}-\mu_j)\mu_j\geq \sum_{k=1}^{N}(\mu_{m+1}\int_{\Omega}|\nabla \tilde{U}_k|^2\,dx-\int_{\Omega}|D^2 \tilde{U}_k|^2).
\end{equation}
Consequently,
\begin{multline*}
  \left(\mu_{m+1}\sum_{j=1}^{n}(\lambda_{n+1}-\lambda_j)+\sum_{j=2}^{m}(\mu_{m+1}-\mu_j)\mu_j\right)(1-a)\\
	\geq \sum_{k=1}^{N}\left((1-a)\mu_{m+1}\lambda_{n+1}+a\left(\lambda_{n+1}-\frac 1 N \sum_{j=1}^{n}(\lambda_{n+1}-\lambda_j)\right)^2-\tilde{\Lambda}_k(a)\right).
\end{multline*}
Moreover,
\begin{equation*}
  (1-a)\sum_j\left(z(z-\lambda_j)_{+}+(z-\mu_j)_{+}\mu_j\right)
	\geq \sum_{k=1}^N\left((1-a)z^2+a\left(z-\frac 1 N \sum_j(z-\lambda_j)_{+}\right)^2-\tilde{\Lambda}_k\right),
\end{equation*}
and in particular for $a=0$
\begin{equation*}
  z\sum_j\left((z-\lambda_j)_{+}-(z-\mu_j)_{+}\right)+\sum_{j=2}^{m}(z^2-\mu_j^2)\geq \sum_j(z^2-\tilde{\Lambda}_j)_{+}.
\end{equation*}
\end{thm}

\begin{proof}
Inequality \eqref{prima} is obtained from \eqref{RieszVersion} with $\omega_j=\lambda_j$, $\psi_j=u_j$, $f_p=\tilde U_k$, and $Q(f,f)=\int_\Omega|\nabla f|^2$, while for \eqref{seconda} we used $\omega_j=\mu_j$, $\psi_j=v_j$, $f_p=\partial_\alpha\tilde U_k$, and then we summed over $\alpha=1,\dots,d$. Moreover, inequality \eqref{prima} also yields
$$
\sum_{j=1}^{n}\left(\lambda_{n+1}-\lambda_j\right)\geq \sum_{k=1}^{N}\left(\lambda_{n+1}-\frac t 2 -\frac 1 {2t}\int_{\Omega}(\Delta \tilde{U}_k)^2\,dx\right),
$$
which coupled with \eqref{seconda} provides the appearance of $\tilde\Lambda_k$.
\end{proof}


\section{The biharmonic Neumann operator}

In this section we focus our attention to the biharmonic Neumann problem \eqref{equazione}, \eqref{NBC}. In particular, the quadratic form \eqref{qf} will be set into $H^2(\Omega)$, for $\Omega\subset \mathbb R^d$ a bounded set with continuous boundary.


Our result is an improvement of the  Kr\"{o}ger-Laptev bound using
a refinement of Young's inequality for real numbers, which
not only improves the estimates for Riesz means and sums, but also
provides a bound on individual eigenvalues.
It will be useful to introduce the following notation:
\begin{equation*}
  m_k:= C_d^2\left(\frac{k}{|\Omega|}\right)^{4/d}, \quad S_k(a):=\frac{\frac{d+4}{d}\frac{1}{k}\sum_{j=1}^k M_j(a)}{m_k}.
\end{equation*}
Note that $m_k$ is the Weyl expression, and the Kr\"{o}ger-Laptev inequality 
is expressed as $S_k\leq 1$. We prove the following refinement of this inequality.

\begin{thm}
For all $k\geq 0$, and for all $a\in(-(d-1)^{-1},1)$, the Neumann eigenvalue $M_{k+1}(a)$ satisfies
\begin{equation*}
  m_k(1-S_k(a)) \geq (\sqrt{M_{k+1}(a)}-\sqrt{m_k})^2,
\end{equation*}
or equivalently
\begin{equation*}
  m_k\left(1-\sqrt{1-S_k(a)}\right)^2\leq M_{k+1}(a) \leq m_k\left(1+\sqrt{1-S_k(a)}\right)^2.
\end{equation*}
\end{thm}
\begin{proof}
The trial functions $f(x)=e^{i p\cdot x}$ are admissible, so choosing them in \eqref{RieszVersion} (see also \cite{Kro,Lap1997}) leads after a calculation to the following bound for the eigenvalues of the Neumann biharmonic operator, where the set $\mathfrak{M}$ is chosen as $\{{ p} \in \mathbb{R}^d\}$ with Lebesgue measure, and $\mathfrak{M_0}$ is the ball of radius $R$ in $\mathbb R^d$ (see
\cite{EHIS,Kro} for details of the calculation):

\begin{equation*}
  \mu_{k+1}R^d-\frac{d}{d+4}R^{d+4}\leq  m_k^{d/4}\left(M_{k+1}(a)-\frac{1}{k}\sum_{i=1}^{k}M_i(a)\right),
\end{equation*}
for all $R>0$.
Putting $R^d=m_k^{d/4}x_k^{d/4}$ with $x_k=\frac{M_{k+1}}{m_k}$ we get the bound
\begin{equation*}
  \frac{d+4}{d}\frac{1}{k}\sum_{i=1}^{k}M_i(a)-m_k\leq m_k\frac{4}{d}\,\left(\frac{d+4}{4}\,x_k-\frac{d}{4}-x_k^{\frac{d+4}{4}}\right).
\end{equation*}
Applying the refinement of Young's inequality given by Lemma \ref{technical_lemma} with $p=d/4$, we obtain
\begin{equation*}
  \frac{d+4}{d}\frac{1}{k}\sum_{i=1}^{k}M_i(a)-m_k\leq -m_k\,(\sqrt{x_k}-1)^2,
\end{equation*}
which strengthens the Kr\"oger-Laptev estimate
\begin{equation*}
  \frac{d+4}{d}\frac{1}{k}\sum_{i=1}^{k}M_i(a)\leq m_k=C_d^2\frac{k^{4/d}}{|\Omega|^{4/d}}
\end{equation*}
and yields the desired bound on $M_{k+1}(a)$.
\end{proof}

\begin{lem}
\label{technical_lemma}
For any $p,x\ge 0$, let $y_p(x)=(p+1)x-p-x^{p+1}$. Then
$$
y_p(x)\le -p(1-\sqrt{x})^2.
$$
\end{lem}

\begin{proof}
From Young's inequality we know that $y_p(x)\le 0$ (see \cite{HaSt16}). The assertion follows from the identity
$$
y_p(x)=-p(1-\sqrt{x})^2+\sqrt{x}y_{2p}(\sqrt{x}).
$$
\end{proof}


\section{One dimensional biharmonic eigenvalue problems}
\label{biharmonic1d}
On the interval $[0,1]$ we consider the fourth-order eigenvalue value problems:
\begin{equation}\label{1d-biharmonic-ev-problems}
\left\{\begin{array}{l}
u^{(4)}_n(x)=\Lambda^{(i,j)}_nu_n(x),\quad x\in(0,1), \\
u^{(i)}(0)=u^{(j)}(0)=u^{(i)}(1)=u^{(j)}(1)=0,
\end{array}\right.
\end{equation}
where $i,j\in\{0,1,2,3\}$, $i\neq j$ and $u_n^{(j)}$ denotes the $j$-th derivative of the function $u_n$. There are six eigenvalue problems beginning with the Dirichlet problem corresponding to $(0,1)$ up to the Neumann problem corresponding to $(2,3)$. It is easy to see that problem \eqref{1d-biharmonic-ev-problems} is a Sturm-Liouville problem for any choice of $i,j$, and in particular the spectrum consists of an increasing sequence of simple eigenvalues (with the possible exception of the kernel) diverging to infinity. In order to further analyze the eigenvalues, we need to study the equation
\begin{equation}\label{1-d-ev-equation}
  \cos\gamma \cosh\gamma=1.
\end{equation}
Let $\gamma_0=0$ and $\gamma_n$ be the positive roots (in increasing order) of \eqref{1-d-ev-equation}. Then
\begin{equation}\label{1st-1d-ev-expansion}
  \gamma_n=\pi\left(n+\frac{1}{2}\right)+(-1)^{n+1}r_n,\quad 0<r_n < \frac{\pi}{2}
\end{equation}
where $r_n$ is strictly decreasing in $n$ and satisfies the following bounds.
\begin{prop}
  For all odd positive integers
\begin{equation}\label{sigma-bound-n-odd}
  \frac{1}{2} {\rm arcsinh}\left(\frac{2}{\cosh\pi\left(n+\frac{1}{2}\right)}\right) \leq r_n\leq \arcsin \left(\frac{1}{\cosh\pi\left(n+\frac{1}{2}\right)}\right),
\end{equation}
and for all even positive integers
\begin{equation}\label{sigma-bound-n-even}
  \arcsin \left(\frac{1}{\cosh\pi\big(n+\frac{1}{2}\big)}\right) \leq r_n\leq \arcsin \left(\frac{2}{\cosh\pi\big(n+\frac{1}{2}\big)}\cdot\frac{1}{1+\sqrt{1-\frac{4}{\cosh\pi(n+\frac{1}{2})}}}\right).
\end{equation}
Therefore, as $n\to\infty$,
\begin{equation}\label{gamma-expansion}
r_n=(-1)^{n+1}\frac{1}{\cosh\pi(n+\frac{1}{2})}+O\left(\frac{1}{\cosh^2\pi(n+\frac{1}{2})}\right).
\end{equation}
\end{prop}

\begin{proof}
  The cosine function is positive between the zeros $\left(2m-\frac{1}{2}\right)\pi$ and  $\left(2m+\frac{1}{2}\right)\pi$, where equation \eqref{1-d-ev-equation} always has two roots by the intermediate value theorem applied to the continuous function $\gamma \mapsto\cos\gamma \cosh\gamma$, since $\cos 2m\pi \cosh 2m\pi= \cosh 2m\pi >1$.
	Therefore, we may label the positive roots $\gamma$ as in \eqref{1st-1d-ev-expansion}, where $r_n$ verifies the condition
\begin{equation}\label{condition_rn}
 1=\sin r_n \cosh \left(\pi\left(n+\frac{1}{2}\right)+(-1)^{n+1} r_n\right),
\end{equation}
from which we easily derive the inequalities \eqref{sigma-bound-n-odd} and \eqref{sigma-bound-n-even}, having the asymptotic expansion \eqref{gamma-expansion} as a consequence. From the equations $\cosh \gamma_{n+1}\sin r_{n+1} = \cosh \gamma_{n}\sin r_{n}$ and $\gamma_{n+1}> \gamma_{n}$ we see that $r_{n}$ is strictly decreasing.
\end{proof}

We now present the spectra and the associated (non-normalized) eigenfunctions of the different eigenvalue problems.

\begin{itemize}
\item {\it Biharmonic Dirichlet eigenvalue problem.}
The eigenfunctions are of the form
\begin{equation*}
  u_n(x)=A\big(\cosh(\gamma_nx)-\cos(\gamma_nx)\big)- \sinh(\gamma_nx)+\sin(\gamma_nx),
\end{equation*}
with $  A=\frac{\sinh(\gamma_n)-\sin(\gamma_n)}{\cosh(\gamma_n)-\cos(\gamma_n)}$ and
\begin{equation*}
  \Lambda^{(0,1)}_n= \gamma_n^4.
\end{equation*}

\item {\it Navier eigenvalue problem.}
The operator is the square of the Dirichlet Laplacian and therefore
\begin{equation*}
  \Lambda^{(0,2)}_n= \pi^4n^4.
\end{equation*}

\item {\it Dirichlet-Neumann mixed eigenvalue problem.}
There is one zero eigenvalue with eigenfunction $u_1(x)=x(1-x)$. For the positive eigenvalues the eigenfunctions are of the form
\begin{equation*}
  u_n(x)=A\big(\cosh(\gamma_nx)-\cos(\gamma_nx)\big)- \sinh(\gamma_nx)-\sin(\gamma_nx),
\end{equation*}
with $A=\frac{\sinh(\gamma_n)+\sin(\gamma_n)}{\cosh(\gamma_n)-\cos(\gamma_n)}$ and
\begin{equation*}
  \Lambda^{(0,3)}_n= \gamma_{n-1}^4.
\end{equation*}

\item {\it Neumann-Dirichlet mixed eigenvalue problem.}
There is one zero eigenvalue with eigenfunction $u_1(x)=1$. For the positive eigenvalues the eigenfunctions are of the form
\begin{equation*}
  u_n(x)=\big(\cosh(\gamma_nx)+\cos(\gamma_nx)\big)- A\big(\sinh(\gamma_nx)-\sin(\gamma_nx)\big),
\end{equation*}
with $A=\frac{\sinh(\gamma_n)-\sin(\gamma_n)}{\cosh(\gamma_n)-\cos(\gamma_n)}$ and
\begin{equation*}
  \Lambda^{(1,2)}_n= \gamma_{n-1}^4.
\end{equation*}

\item {\it Kuttler-Sigillito eigenvalue problem.}
The operator is the square of the Neumann Laplacian and therefore
\begin{equation*}
  \Lambda^{(1,3)}_n= \pi^4(n-1)^4.
\end{equation*}

\item {\it Biharmonic Neumann eigenvalue problem.}
The eigenvalue $0$ has multiplicity $2$ with corresponding eigenfunctions $1,x$. For the positive eigenvalues the eigenfunctions are of the form
\begin{equation*}
  u_n(x)=A\big(\cosh(\gamma_nx)+\cos(\gamma_nx)\big)- \big(\sinh(\gamma_nx)-\sin(\gamma_nx)\big),
\end{equation*}
with $A=\frac{\sinh(\gamma_n)-\sin(\gamma_n)}{\cosh(\gamma_n)-\cos(\gamma_n)}$ and
\begin{equation*}
 \Lambda^{(2,3)}_1,\Lambda^{(2,3)}_2=0,\quad \Lambda^{(2,3)}_n= \gamma_{n-2}^4, \quad n\geq 3.
\end{equation*}
\end{itemize}

\begin{rem}
We note that the Dirichlet-Neumann and the Neumann-Dirichlet eigenvalue problems have not been described in Section \ref{sec:2} and in fact they have a completely different nature from the Dirichlet, Navier, Kuttler-Sigillito and Neumann problems. They can not be associated with the quadratic form \eqref{qf}, however, they can be understood as a ``reduction'' of a buckling-type eigenvalue problem of sixth order which has a  standard variational formulation, namely problem
\begin{equation}
\begin{cases}\label{buck6}
u^{(6)}(x)=\omega u''(x)\,,& x\in(0,1),\\
u(0)=u'''(0)=u^{(4)}(0)=u(1)=u'''(1)=u^{(4)}(1)=0.
\end{cases}
\end{equation}
The weak formulation of this problem reads
$$
\int_0^1 u'''(x)\phi'''(x)dx=\omega\int_0^1 u'(x)\phi'(x)dx\,,\ \ \ \forall\phi\in H^3((0,1))\cap H^1_0((0,1)),
$$
in the unknowns $\omega\in\mathbb R$, $u\in H^3((0,1))\cap H^1_0((0,1))$. It is standard to recast this problem to an eigenvalue problem for a compact self-adjoint operator on a Hilbert space. 

It is easy to prove that any solution of $u^{(6)}(x)=\omega u''(x)$ is of the form 
$$
u(x)=a_0+a_1x+a_2\sin(\omega^{1/4}x)+a_3\cos(\omega^{1/4}x)+a_4 e^{\omega^{1/4}x}+a_5 e^{-\omega^{1/4}x},
$$
for some $a_0,...,a_5\in\mathbb R$. Imposing boundary conditions it is possible to prove that $\omega_1=0$ is an eigenvalue of multiplicity one with corresponding (non-normalized) eigenfunction $x(1-x)$, while all the other eigenvalues $\omega_n$, $n\geq 2$, are simple and positive, and are given implicitly by the equation $\cos(\omega^{1/4})\cosh(\omega^{1/4})=1$. This means that $\omega_n=\gamma_{n-1}^4$, i.e., the eigenvalues of \eqref{buck6} coincide with $\Lambda_n^{(0,3)}$ and $\Lambda_n^{(1,2)}$.

Also the eigenfunctions coincide. We claim that an eigenpair $(u,\omega)$ of \eqref{buck6} is also an eigenpair of the Dirichlet-Neumann problem, and vice-versa. In fact, let $(u,\omega)$ be an eigenpair of \eqref{buck6}. Then we have that $(u^{(4)}(x)-\omega u(x))''=0$ for $x\in(0,1)$ and from the boundary conditions we also have that $u^{(4)}(0)-\omega u(0)=u^{(4)}(1)-\omega u(1)=0$, thus $u^{(4)}(x)-\omega u(x)=0$ for $x\in(0,1)$. Moreover, $u'''(0)=u'''(1)=0$. This implies that $u$ is a solution to the Dirichlet-Neumann problem. Vice-versa, let $(u,\omega)$ an eigenpair of the Dirichlet-Neumann problem. Thus $u^{(6)}(x)=\omega u''(x)$ for $x\in(0,1)$, and clearly $u(0)=u'''(0)=u(1)=u'''(1)=0$. Moreover $u^{(4)}(0)=\omega u(0)=0$ and $u^{(4)}(1)=\omega u(1)=0$. 

Analogous arguments allow to deduce that, given an eigenpair $(u,\omega)$ of problem \eqref{buck6}, then $(u'',\omega)$ is an eigenpair of the Neumann-Dirichlet problem. Vice-versa, given an eigenpair $(v,\omega)$ of the Neumann-Dirichlet problem, then $(u,\omega)$ is an eigenpair of problem \eqref{buck6},  where $u$ is the solution to the boundary value problem $u''(x)=v(x)$ for $x\in (0,1)$, $u(0)=u(1)=0$.



\end{rem}




Summarizing, for $n\geq 1$ we have 
\begin{eqnarray*}
  \Lambda^{(0,1)}_n&=& \gamma_n^4\\
  \Lambda^{(0,2)}_n &=& \pi^4n^4 \\
  \Lambda^{(0,3)}_n &=& \gamma_{n-1}^4\\
\Lambda^{(1,2)}_n&=& \gamma_{n-1}^4 \\
  \Lambda^{(1,3)}_n &=& \pi^4(n-1)^4 \\
  \Lambda^{(2,3)}_n &=& \gamma_{n-2}^4\\
\end{eqnarray*}
with the convention $\gamma_{-1}=0$. Since $n<\gamma_n<n+1$, the spectra are in decreasing order and the eigenvalues of two ``neighbored'' operators in the table are interlacing with strict inequalities for all positive eigenvalues (with the only exception $\Lambda^{(0,3)}_n=\Lambda^{(1,2)}_n$). In particular, for all positive integers $n$ we have the following identities
\begin{equation*}
  \Lambda^{(0,1)}_n=\Lambda^{(0,3)}_{n+1}=\Lambda^{(1,2)}_{n+1}=\Lambda^{(2,3)}_{n+2}.
\end{equation*}
We note that the Neumann eigenvalues satisfy the sharp Weyl-type bound of the form $  \Lambda^{(2,3)}_n\leq \pi^4(n-1)^4$ and not
$  \Lambda^{(2,3)}_n\leq \pi^4(n-2)_{+}^4$ where the shift is made by the dimension of the kernel.

With respect to semiclassical limits, while there is no two-term asymptotic expansion of the form \eqref{semiclassicalcounting} for the counting function $N(z)$, the expansion \eqref{semiclassicalriesz} for the Riesz means $R_1(z)$ is still valid. This is a corollary of the following asymptotically sharp two-term upper and lower bounds. We shall denote by $R_1^{(i,j)}(z)$ the Riesz means corresponding to the eigenvalues $\Lambda_n^{(i,j)}$, $i,j\in\{0,1,2,3\}$, $i\ne j$. A crucial ingredient in the proof of the following theorem is Lemma \ref{lemmaonedim} on polynomial bounds for one dimesional Riesz means, which is proved at the end of this section.

\begin{thm}\label{riesz-1-d}
For any $z>0$ we have the following inequalities
\begin{itemize}

\item {\it Biharmonic Dirichlet eigenvalue problem.}
\begin{multline*}
\frac{4}{5\pi}z^{\frac{5}{4}}-z-\frac{11\pi}{6}z^{\frac{3}{4}}-\frac{3\pi^2}{2}z^{\frac{1}{2}}-\frac{127\pi^3}{240}z^{\frac{1}{4}}-c\\
\leq\sum_{n\geq 1}(z-\Lambda_n^{(0,1)})_+\\
\leq\frac{4}{5\pi}z^{\frac{5}{4}}-z+\frac{\pi}{6}z^{\frac{3}{4}}+\frac{3\pi^2}{2}z^{\frac{1}{2}}+\frac{1\pi^3}{30}z^{\frac{1}{4}}+\frac{1\pi^4}{8}+c.
\end{multline*}
\item {\it Navier eigenvalue problem.}
\begin{multline*}
\frac{4}{5\pi}z^{\frac{5}{4}}-\frac{1}{2}z-\frac{\pi}{3}z^{\frac{3}{4}}\\
\leq\sum_{n\geq 1}(z-\Lambda_n^{(0,2)})_+\\
\leq\frac{4}{5\pi}z^{\frac{5}{4}}-\frac{1}{2}z+\frac{\pi}{6}z^{\frac{3}{4}}+\frac{\pi^2}{12}z^{\frac{1}{2}}.
\end{multline*}
\item {\it Dirichlet-Neumann and Neumann-Dirichlet mixed eigenvalue problem.}
\begin{multline*}
\frac{4}{5\pi}z^{\frac{5}{4}}-\frac{11\pi}{6}z^{\frac{3}{4}}-\frac{3\pi^2}{2}z^{\frac{1}{2}}-\frac{127\pi^3}{240}z^{\frac{1}{4}}-c\\
\leq\sum_{n\geq 1}(z-\Lambda_n^{(0,3)})_+=\sum_{n\geq 1}(z-\Lambda_n^{(1,2)})_+\\
\leq\frac{4}{5\pi}z^{\frac{5}{4}}+\frac{\pi}{6}z^{\frac{3}{4}}+\frac{3\pi^2}{2}z^{\frac{1}{2}}+\frac{1\pi^3}{30}z^{\frac{1}{4}}+\frac{1\pi^4}{8}+c.
\end{multline*}
\item {\it Kuttler-Sigillito eigenvalue problem.}
\begin{multline*}
\frac{4}{5\pi}z^{\frac{5}{4}}+\frac{1}{2}z-\frac{\pi}{3}z^{\frac{3}{4}}\\
\leq\sum_{n\geq 1}(z-\Lambda_n^{(1,3)})_+\\
\leq\frac{4}{5\pi}z^{\frac{5}{4}}+\frac{1}{2}z+\frac{\pi}{6}z^{\frac{3}{4}}+\frac{\pi^2}{12}z^{\frac{1}{2}}.
\end{multline*}
\item {\it Biharmonic Neumann eigenvalue problem.}
\begin{multline*}
\frac{4}{5\pi}z^{\frac{5}{4}}+z-\frac{11\pi}{6}z^{\frac{3}{4}}-\frac{3\pi^2}{2}z^{\frac{1}{2}}-\frac{127\pi^3}{240}z^{\frac{1}{4}}-c\\
\leq\sum_{n\geq 1}(z-\Lambda_n^{(2,3)})_+\\
\leq\frac{4}{5\pi}z^{\frac{5}{4}}+z+\frac{\pi}{6}z^{\frac{3}{4}}+\frac{3\pi^2}{2}z^{\frac{1}{2}}+\frac{1\pi^3}{30}z^{\frac{1}{4}}+\frac{1\pi^4}{8}+c.
\end{multline*}
\end{itemize}
Here $c\in]2,3[$ is defined by \eqref{c}.
\end{thm}

\begin{proof}
Let us start with the Dirichlet eigenvalues $\Lambda^{(0,1)}_n=\gamma_n^4$. We have
$$
R_1^{(0,1)}(z)=\sum_{n\geq 1}\left(z-\left(\pi\left(n+\frac 1 2\right)+(-1)^{n+1}r_n\right)^4\right)_+.
$$
We set $R:=z^{\frac{1}{4}}\pi^{-1}$ and $\rho_n=r_n\pi^{-1}$. With this notation we ave 
\begin{equation*}
  \sum_{n\geq 1}(z-\gamma_n^4)_{+}=\pi^{4}\sum_{n\geq 1}(R^4-\gamma_n^4\pi^{-4})_{+}=\pi^{4}\sum_{n=1}^N\left(R^4-\left(n+\frac{1}{2}+(-1)^{n+1}\rho_n\right)^4\right)
\end{equation*}
for some $N$ satisfying $R-2< N\leq R$. Expanding the sum we get
\begin{equation*}
\begin{split}
   \sum_{n=1}^{N}\left(R^4-\left(n+\frac{1}{2}+(-1)^{n+1}\rho_n\right)^4\right) & =\sum_{n=1}^{N}\left(R^4-\left(n+\frac{1}{2}\right)^4\right) \\
     & -4\sum_{n=1}^{N}(-1)^{n+1}\left(\rho_n\left(n+\frac{1}{2}\right)^3+\rho_n^3\left(n+\frac{1}{2}\right)\right)\\
     &-\sum_{n=1}^{N}\left(6\rho_n^2\left(n+\frac{1}{2}\right)^2+\rho_n^4\right).\\
\end{split}
\end{equation*}
The sums containing $\rho_n$ lead to absolutely converging series as $N$ tends to infinity. In order to have explicit bounds on these sums we need simpler bounds on $r_n$. From condition \eqref{condition_rn} we deduce that, for odd $n$
$$
\sin r_n=\frac{1}{\cosh \left(\pi\left(n+\frac{1}{2}\right)+r_n\right)}\leq \frac{1}{\cosh \pi\left(n+\frac{1}{2}\right)}\leq 2 e^{-\pi n}.
$$
Using the fact that $\sin x\geq \frac{2}{\pi}x$ for alla $x\in[0,\frac{\pi}{2}]$ we deduce that $r_n\leq \pi e^{-\pi n}$. For even $n$ we use the fact that $0<r_n<\frac{\pi}{2}$ and deduce that 
$$
\sin r_n=\frac{1}{\cosh \left(\pi\left(n+\frac{1}{2}\right)-r_n\right)}\leq \frac{1}{\cosh \left(\pi\left(n+\frac{1}{2}\right)-\frac{\pi}{2}\right)}=\frac{1}{\cosh \left(\pi n\right)}\leq 2 e^{-\pi n}.
$$
As in the odd case, we deduce that $r_n\leq \pi e^{-\pi n}$. Then
\begin{multline}\label{c}
\left|\sum_{n=1}^{N}4(-1)^{n+1}\left(\rho_n\left(n+\frac{1}{2}\right)^3+\rho_n^3\left(n+\frac{1}{2}\right)\right)+\left(6\rho_n^2\left(n+\frac{1}{2}\right)^2+\rho_n^4\right)\right|\\
\leq \sum_{n=1}^{\infty}4\left(\pi e^{-\pi n}\left(n+\frac{1}{2}\right)^3+\pi^3 e^{-3\pi n}\left(n+\frac{1}{2}\right)\right)+6\pi^2e^{-2\pi n}\left(n+\frac{1}{2}\right)^2+\pi^4e^{-4\pi n}=:c,
\end{multline}
where $c$ can be explicitly computed ($c\approx 2.51272...$). In particular the previous estimates imply that $2<c<3$. Therefore 
\begin{equation*}
  -c\leq \sum_{n\geq 1}\left(z-\gamma_n^4\right)_{+}-\pi^{4}\sum_{n=1}^{N}\left(R^4-\left(n+\frac{1}{2}\right)^4\right)\leq c.
\end{equation*}
We shall prove in Lemma \ref{lemmaonedim} asymptotically sharp upper and lower bounds on $\sum_{n=1}^{N}\left(R^4-n^4\right)$, see \eqref{onedim1}, and on $\sum_{n=1}^{N}\left(R^4-\left(n+\frac{1}{2}\right)^4\right)$, see \eqref{onedim2}. In particular, upper and lower bounds on $R_1^{(0,1)}(z)$ follow from \eqref{onedim2} .

The bounds on $R_1^{(2,3)}(z)$ follow from those on $R_1^{(0,1)}(z)$ by observing that $R_1^{(2,3)}(z)=R_1^{(0,1)}(z)+2z$, where the additional term is due to the kernel. In the same way, the bounds on $R_1^{(0,3)}(z)$ and $R_1^{(1,2)}(z)$ follow from those on $R_1^{(0,1)}(z)$ by observing that $R_1^{(0,3)}(z)=R_1^{(1,2)}(z)=R_1^{(0,1)}(z)+z$.

The Riesz means $R_1^{(0,2)}(z)$ for the Navier problem is instead explicitly computable and
$$
R_1^{(0,2)}(z)=\sum_n\left(z-\pi^4 n^4\right)_+=\pi^4\sum_n\left(R^4-n^4\right)_+.
$$
Upper and lower bounds follow then from \eqref{onedim1}. Finally, upper and lower bounds for $R_1^{(1,3)}(z)$ follow from those on $R_1^{(0,2)}(z)$ by noting that $R_1^{(1,3)}(z)=R_1^{(0,2)}(z)+z$.
\end{proof}

From Theorem \ref{riesz-1-d} we deduce the following
\begin{cor}
The following expansion holds
$$
R_1^{(i,j)}(z)=\frac 4 {5\pi}z^{\frac 5 4}+c_1^{(i,j)} z+O(z^{\frac 3 4}),
$$
where $c_1^{(i,j)}=\frac{i+j-3}{2}$. 
\end{cor}

We conclude by proving the following polynomial upper and lower bounds on one dimensional Riesz means.

\begin{lem}\label{lemmaonedim}
For all $R\geq 0$ the following inequalities hold:
\begin{equation}\label{onedim1}
-\frac{1}{3}R^3
\leq\sum_{n\geq 1}\left(R^4-n^4\right)_+-\frac{4}{5}R^5+\frac{1}{2}R^4
\leq \frac{1}{6}\,R^3+\frac{1}{12}\,R^2.
\end{equation}
and
\begin{equation}\label{onedim2}
-\frac{11}{6}R^3-\frac{3}{2}R^2-\frac{127}{240}R
\leq\sum_{n\geq 1}\left(R^4-\left(n+\frac{1}{2}\right)^4\right)_+-\frac{4}{5}R^5+R^4
\leq \frac{1}{6}R^3+\frac{3}{2}R^2+\frac{1}{30}R+\frac{1}{8}.
\end{equation}
In particular \eqref{onedim2} holds if we replace $\sum_{n\geq 1}\left(R^4-\left(n+\frac{1}{2}\right)^4\right)_+$ by $\sum_{n=1}^N\left(R^4-\left(n+\frac{1}{2}\right)^4\right)$ with $N\in\mathbb N$, $R-2\leq N\leq R$.
\end{lem}
\begin{proof}
We start by proving \eqref{onedim1} for $R\geq 1$. We have that
$$
\sum_{n\geq 1}\left(R^4-n^4\right)_+=\sum_{n=1}^N \left(R^4-n^4\right)
$$
where $N=[R]$ and therefore $R-1<N \leq R$. Expanding and re-arranging the sum in a suitable way we get
$$
\sum_{n=1}^N \left(R^4-n^4\right)-\frac{4}{5}R^5+\frac{1}{2}R^4=-\frac{1}{5}N^5-\frac{1}{2}N^4-\frac{1}{3}N^3+\left(\frac{1}{30}+R^4\right)N+\frac{1}{2}R^4-\frac{4}{5}R^5.
$$
We set
$$
f_R(x):=-\frac{1}{5}x^5-\frac{1}{2}x^4-\frac{1}{3}x^3+\left(\frac{1}{30}+R^4\right)x+\frac{1}{2}R^4-\frac{4}{5}R^5
$$
and we estimate its maximum and minimum for $x\in[R-1,R]$, $R\geq 1$. We have that $f_R''(x)=-2x(x+1)(2x+1)<0$ for all $x\geq 1$. We also note that $f_R'(R-\frac{1}{2})=\frac{R^2}{2}-\frac{7}{240}>0$ and $f_R'(R-\frac{1}{3})=-\frac{2}{3}R^3+\frac{1}{3}R^2+\frac{4}{27}R-\frac{13}{810}<0$, for all $R\geq 1$. Hence the maximum is attained in between these two points at some $N_0$. By concavity
\begin{multline*}
  f_R(N_0)=f_R\left(R-\frac{1}{2}\right)+\int_{R-1/2}^{N_0}f_R'(y)\,dy\leq f_R\left(R-\frac{1}{2}\right)+\left(N_0-R+\frac{1}{2}\right)f_R'\left(R-\frac{1}{2}\right)\\ \leq f_R\left(R-\frac{1}{2}\right)+\frac{1}{6} f_R'\left(R-\frac{1}{2}\right)
\end{multline*}
which implies
\begin{equation*}
  f_R(N_0)\leq \frac{1}{6}R^3+\frac{1}{12}R^2-\frac{7}{240}R-\frac{7}{1440}\leq \frac{1}{6}R^3+\frac{1}{12}R^2.
\end{equation*}
We deduce the bound
$$
  \sum(R^4-n^4)_{+}\leq \frac{4}{5}\,R^5 -\frac{1}{2}\,R^4+\frac{1}{6}\,R^3+\frac{1}{12}\,R^2
$$
which holds for all $R\geq 0$.

By concavity we also deduce that
$$
f_R(x)\geq \min\{f_R(R-1),f_R(R)\}=\min\left\{-\frac{1}{3}R^3+\frac{1}{30}R,-\frac{1}{3}R^3+\frac{1}{30}R\right\}=-\frac{1}{3}R^3+\frac{1}{30}R,
$$
for all $R\geq 1$. This implies the bound
\begin{equation}\label{Riesz-mean-n-to-4-upper-bound}
  \sum(R^4-n^4)_{+}\geq \frac{4}{5}\,R^5 -\frac{1}{2}\,R^4-\frac{1}{3}\,R^3
\end{equation}
which holds for all $R\geq 0$. This concludes the proof of \eqref{onedim1}.


It is possible to obtain upper and lower bounds for $\sum_{n\geq 1}\left(R^4-\left(n+\frac{1}{2}\right)^4\right)_+$ as in the proof of \eqref{onedim1}, however computations become involved. Instead, we proceed as follows. We have
$$
\sum_{n\geq 1}\left(R^4-\left(n+\frac{1}{2}\right)^4\right)_+=\sum_{n=1}^N \left(R^4-\left(n+\frac{1}{2}\right)^4\right)
$$
where $N\in\mathbb N$ satisfies $R-2\leq N\leq R$. Expanding and re-arranging the sum in a suitable way we get
\begin{multline}\label{Riesz-mean-n-plus-one-half-to-4}
\sum_{n=1}^{N}(R^4-(n+\frac{1}{2})^4)\\
=\frac{4}{5}\,R^5 -R^4+(\frac{1}{6}-2t^2)R^3+(2t^3-\frac{t}{2})R^2+(\frac{t}{2}-t^4-\frac{7}{240})R+\frac{t^5}{5}-\frac{t^3}{6}+\frac{7t}{240}+\frac{1}{16},
\end{multline}
where $t:=R-1-N$ so that $t\in[-1,1]$. Minimizing and maximizing each coefficient we find
\begin{multline*}
  -\frac{11}{6} \leq  \frac{1}{6}-2t^2\leq \frac{1}{6}, \ \ 
  -\frac{3}{2} \leq  2t^3-\frac{t}{2}\leq \frac{3}{2},\\ 
  -\frac{127}{240} \leq \frac{t}{2}-t^4-\frac{7}{240}\leq \frac{1}{30}, \ \ 
   0 \leq \frac{t^5}{5}-\frac{t^3}{6}+\frac{7t}{240}+\frac{1}{16}\leq \frac{1}{8} .
\end{multline*}
This implies that
$$
-\frac{11}{6}R^3-\frac{3}{2}R^2-\frac{127}{240}R
\leq\sum_{n=1}^N\left(R^4-\left(n+\frac{1}{2}\right)^4\right)-\frac{4}{5}R^5+R^4
\leq \frac{1}{6}R^3+\frac{3}{2}R^2+\frac{1}{30}R+\frac{1}{8}.
$$
This concludes the proof of the lemma.
\end{proof}



\section*{Acknowledgements}

This work was supported by the SNSF project ``Bounds for the Neumann and Steklov eigenvalues of the biharmonic operator'', SNSF grant number 200021\_178736.  Most of the research in this paper was carried out while the first author held a
post-doctoral position at \'Ecole Polytechnique F\'ed\'erale de Lausanne within the scope of this project.
The second author expresses his gratitude to \'Ecole Polytechnique F\'ed\'erale de Lausanne for the hospitality that helped the development of this paper. 
The first author is member of the Gruppo Nazionale per l'Analisi Ma\-te\-ma\-ti\-ca, la Probabilit\`a e le loro Applicazioni (GNAMPA) of the I\-sti\-tuto Naziona\-le di Alta Matematica (INdAM). The second author is member of the Gruppo Nazionale per le Strutture Algebriche, Geometriche e le loro Applicazioni (GNSAGA) of the I\-sti\-tuto Naziona\-le di Alta Matematica (INdAM).








\bibliographystyle{amsplain}
\bibliography{bibliography}

\end{document}